\documentclass[12pt]{amsart}
\usepackage{amsmath,amssymb,amsthm,mathtools}
\usepackage{color}
\usepackage[OT2,T1]{fontenc}
\DeclareSymbolFont{cyrletters}{OT2}{wncyr}{m}{n}
\DeclareMathSymbol{\Sha}{\mathalpha}{cyrletters}{"58}
\usepackage{hyperref}
\hypersetup{hidelinks}

\newtheorem{theorem}{Theorem}[section]
\newtheorem{corollary}[theorem]{Corollary}
\newtheorem{lemma}[theorem]{Lemma}
\newtheorem{proposition}[theorem]{Proposition}
\newtheorem{question}{Question}
\theoremstyle{remark}
\newtheorem{remark}[theorem]{Remark}

\newcommand{\Q}{\mathbb{Q}}
\newcommand{\Z}{\mathbb{Z}}
\renewcommand{\C}{\mathbb{C}}
\newcommand{\Tam}{\mathrm{Tam}}
\newcommand{\Ht}{\mathrm{ht}}
\newcommand{\Prob}{\mathbf P}
\newcommand{\E}{\mathbf E}
\newcommand{\Var}{\mathbf{Var}}
\newcommand{\Cov}{\mathbf{Cov}}
\newcommand{\Corr}{\mathbf{Corr}}
\newcommand{\1}{\mathbf 1}
\newcommand{\NN}{\mathcal N}
\newcommand{\KK}{\mathcal K}
\newcommand{\Reals}{\operatorname{Re}}
\title[Local Factors in the BSD Conjecture: A Unified Statistical View]{Local Factors in the BSD Conjecture: A Unified Statistical View}
\author[Angdinata, Lau, Ono, Swaminathan, and Vemulapalli]{David Kurniadi Angdinata, Kenny Lau, Ken Ono, Ashvin Swaminathan, and Sameera Vemulapalli}
\email{davidang@axiommath.ai, kenny@axiommath.ai, ken@axiommath.ai, ashvin@axiommath.ai, vemulapalli@math.harvard.edu}
\dedicatory{In memory of Takashi Ono}

\begin{document}
\begin{abstract}
For elliptic curves $E/\Q$ in short Weierstrass form
\[
E=E(a_4,a_6): y^2=x^3+a_4x+a_6,
\]
we derive a multivariable Euler product generating function which encodes the Tamagawa product $\Tam(E)=\prod_p c_p(E)$. Using this generating function, we compute limiting distributions, exact covariances, and moment and tail bounds for four important statistics on the Tamagawa number; for example we show that more than half of all curves in this height ordering have trivial Tamagawa product, about $42.2\%$ have exactly one prime with nontrivial local Tamagawa number, and only about $6.8\%$ have exactly two such primes. The product is obtained by specializing an Euler product indexed by local reduction data that we derive from Tate's algorithm. The results in this paper were autoformalized in Lean by AxiomProver.
\end{abstract}

\maketitle

\section{Introduction and statement of results}

We work with short Weierstrass models
\begin{equation}\label{eq:model}
E=E(a_4,a_6):\qquad y^2=x^3+a_4x+a_6,
\end{equation}
where $a_4,a_6\in\Z$ and
\[
\Delta(E)=-16(4a_4^3+27a_6^2)\neq 0.
\]
We order these curves by their naive height
\begin{equation}\label{eq:height}
\Ht(E):=\max\{4|a_4|^3,27a_6^2\},
\end{equation}
and we shall employ the corresponding counting function
\begin{equation}\label{eq:NX}
\NN(X):=\#\{(a_4,a_6)\in\Z^2:\Ht(E(a_4,a_6))\le X,\ \Delta(E(a_4,a_6))\ne0\}.
\end{equation}
Thus, $\NN(X)$ counts integral short Weierstrass models, not $\Q$-isomorphism classes.

For each prime $p$, we have the \emph{local Tamagawa number}
\begin{equation}\label{eq:Tam-local}
c_p(E)=[E(\Q_p):E_0(\Q_p)],
\end{equation}
where $E_0(\Q_p)$ is the subgroup of finite index consisting of the points whose reduction modulo $p$, computed on a $p$-minimal Weierstrass model, is a nonsingular point of the reduced curve. Using these local factors, we define the \emph{Tamagawa product}
\begin{equation}\label{eq:Tam-product}
\Tam(E):=\prod_p c_p(E).
\end{equation}
Only finitely many factors $c_p$ differ from $1$, and so this product is a positive integer.

The primary contribution is to give the distribution and correlations of several statistics associated to the Tamagawa product. For example we show that more than half of all curves in this height ordering have trivial Tamagawa product, about $42.2\%$ have exactly one prime with nontrivial local Tamagawa number, and only about $6.8\%$ have exactly two such primes. See Subsection~\ref{subsec:numerics} for more statistics of Tamagawa products derived from the theorems in this paper. 

The Tamagawa product appears on the algebraic side of the Birch--Swinnerton-Dyer conjecture \cite{BSD2}. If $r=\operatorname{rank}E(\Q)$, then BSD predicts that
\begin{equation}\label{eq:BSD}
\frac{L^{(r)}(E,1)}{r!}
=
\frac{|\Sha(E)|\,\Omega_E\,R_E\,\prod_{p\mid N_E}c_p(E)}{|E(\Q)_{\mathrm{tors}}|^2}.
\end{equation}
In \eqref{eq:BSD}, $L(E,s)$ denotes the Hasse--Weil $L$-function and $N_E$ the conductor. The remaining quantities are the Tate--Shafarevich group $\Sha(E)$, the real period $\Omega_E$, and the regulator $R_E$. The conductor $N_E$ and the Tamagawa product are both determined by the local reduction of $E$. The conductor records the type and severity of the bad primes, whereas $\Tam(E)$ records the sizes of the corresponding component groups.

Another important statistic is the proportion of elliptic curves with fixed Tamagawa product $m$. In recent work, Griffin, Tsai, and one of the authors \cite{GOT} proved, for every $m\ge 1$, a formula for the natural density
\[
P_{\Tam}(m):=\lim_{X\to\infty}\frac{\#\{E:\Ht(E)\le X,\ \Tam(E)=m\}}{\NN(X)}.
\]
These densities exist, and they are the coefficients of the Dirichlet series
\begin{equation}\label{eq:GOT-Dirichlet}
L_{\Tam}(s):=\sum_{m\ge 1}\frac{P_{\Tam}(m)}{m^s},
\end{equation}
which factors as an Euler product built from explicit local densities $\delta_p(t)$ (see Theorem~\ref{thm:GOT-local}).

We consider four statistics associated with the Tamagawa product. The first is the number of primes that contribute nontrivially to the Tamagawa product. This number need not equal the number of primes of bad reduction; indeed, it is smaller for a density-one subfamily.

To this end, for an elliptic curve $E/\Q$ in the family \eqref{eq:model}, we define
\begin{equation}\label{eq:omegaE}
\omega_{\Tam}(E):=\#\{p\text{ prime}: c_p(E)>1\}.
\end{equation}
Clearly, we have
\[
\omega_{\Tam}(E)\le \omega(N_E),
\]
where $\omega(N_E)$ denotes the number of distinct prime divisors of the conductor. More generally, for an integer $n$ we write $\omega(n)$ for the number of distinct prime divisors of $n$ and $\Omega(n)$ for the number of prime divisors counted with multiplicity. We consider the following four questions in arithmetic statistics.

\begin{question}\label{q:topic1}
How does $\omega_{\Tam}(E)$ vary as the height of $E$ grows?
\end{question}

\begin{question}\label{q:topic2}
What is the joint behavior of $\omega_{\Tam}(E)$ and $\Tam(E)$?
\end{question}

\begin{question}\label{q:topic3}
What can one say about the arithmetic of the integer $\Tam(E)$ itself? In particular, what can be established about the statistics
\[
\Omega(\Tam(E)),\qquad \omega(\Tam(E)),\qquad v_\ell(\Tam(E))\quad (\ell\text{ prime}),
\]
and how are they correlated among themselves?
\end{question}

\begin{question}\label{q:topic4}
How rapidly do the probabilities $P_{\Tam}(m)$ decay as $m\to\infty$? Equivalently, what bounds can be obtained, as a function of $M$, for the tails
\[
\sum_{m\ge M}P_{\Tam}(m)?
\]
\end{question}

In this paper we obtain the following results addressing the four questions above.

\subsection{Statement of results}\label{sec:statement-of-results}

We shall derive the arithmetic statistics of the local invariants considered here from a single Euler product indexed by \emph{local reduction data}. Let $\KK$ denote the set of local reduction data used in this paper (see \cite{SilvermanAdvanced, Tate} for background). For each prime $p$, write
\begin{equation}\label{eq:TateLocalDatum}
\tau_p(E)\in \KK
\end{equation}
for the local reduction datum of $E$ at $p$. An element of $\KK$ records the Kodaira symbol, the split/nonsplit data, and the group of $\mathbb{F}_p$-components. Thus, there are maps
\[
\kappa:\KK\to\{\text{Kodaira symbols}\},
\qquad
c:\KK\to\Z_{\ge1},
\]
where $\kappa$ is the Kodaira-symbol map and $c$ is the local Tamagawa-number map. We write
\[
\KK=\KK_{\mathrm{good}}\cup\KK_{\mathrm{mult}}\cup\KK_{\mathrm{add}}
\]
according to the good, multiplicative, and additive reduction classes of the underlying Kodaira symbol. Namely, we have
\[
\KK_{\mathrm{good}}=\{I_0\},\qquad
\KK_{\mathrm{mult}}=\{I_n^{\mathrm{sp}},\,I_n^{\mathrm{ns}}:n\ge 1\},
\]
\[
\kappa(\KK_{\mathrm{add}})=\{II,\,III,\,IV,\,I_0^\ast,\,I_n^\ast\ (n\ge 1),\,IV^\ast,\,III^\ast,\,II^\ast\}.
\]
In the multiplicative case, the displayed labels include the split/nonsplit distinction.  For additive reduction, a single Kodaira symbol may occur with several component groups. With this notation, we have
\[
c_p(E)=c(\tau_p(E))
\]
for a function $c:\KK\to\Z_{\ge1}$; see \cite[Chapter IV, \S9]{SilvermanAdvanced}, \cite{Tate}, and the local cases in Tables~A1-A3 of \cite{GOT}, with the correction at $p=3$ described in Remark~\ref{rem:three-adic-correction}. Correspondingly, for each local reduction datum $K\in\KK$, we let
\[
\delta_p(K):=\mu_p(\tau_p^{-1}(K))
\]
denote the proportion of short Weierstrass models whose $p$-minimal model has local reduction datum $K$, where $\mu_p$ is the normalized Haar measure on the space $\Z_p^2$ of coefficient pairs $(a_4,a_6)$. Similarly, for a Kodaira symbol $I$, let $\delta_p(I)$ denote the proportion of short Weierstrass models whose $p$-minimal model has Kodaira symbol $I$. The scalar densities of \cite{GOT} are aggregated from these by \begin{equation}\label{deltapt}
    \delta_p(t):=\sum_{c(K)=t}\delta_p(K).
    \end{equation}

We now let
\begin{equation}\label{eq:K0-def}
\KK_0:=\{I_0,\,I_1^{\mathrm{sp}},\,I_1^{\mathrm{ns}}\}
\qquad\text{and}\qquad
\beta_p:=\delta_p(I_0)+\delta_p(I_1^{\mathrm{sp}})+\delta_p(I_1^{\mathrm{ns}}).
\end{equation}
Then, as shown in Section~\ref{sec:local_densities}, for $p\ge 5$, we have
\begin{equation}\label{eq:K0-asymptotic-intro}
\delta_p(I_0)=1-O(p^{-1}),\qquad
\delta_p(I_1^{\mathrm{sp}})\asymp \delta_p(I_1^{\mathrm{ns}})\asymp p^{-1},\qquad
\sum_{K\notin \KK_0}\delta_p(K)\ll p^{-2}.
\end{equation}
In particular, $\beta_p=1-O(p^{-2})$. All three elements of $\KK_0$ have local Tamagawa number $1$, so they contribute trivially to the statistics considered here. For each local reduction datum $K\in\KK\setminus\KK_0$, define
\begin{equation}\label{eq:omegaKE}
\omega_K(E):=\#\{p\text{ prime}:\tau_p(E)=K\}.
\end{equation}
Then $\sum_p\delta_p(K)<\infty$. By contrast, $\sum_p\delta_p(K)=\infty$ for every $K\in\KK_0$, so the counts $\omega_K(E)$ with $K\in\KK_0$ need not admit a limiting distribution under naive-height ordering. Since the local reduction datum records $c_p(E)$, the statistic $\omega_{\Tam}(E)$ of \eqref{eq:omegaE} decomposes as
\[
\omega_{\Tam}(E)=\sum_{\substack{K\in\KK\setminus\KK_0\\ c(K)>1}}\omega_K(E).
\]

To answer Questions~\ref{q:topic1}--\ref{q:topic4}, we define a key multivariable generating function. For a finite set $\Lambda\subseteq \KK\setminus\KK_0$ of local reduction data and a finite set of primes $\Pi$, we let
\[
\mathbf u=(u_K)_{K\in\Lambda},
\qquad \mathbf z=(z_\ell)_{\ell\in\Pi},
\qquad \mathbf z^{\mathbf j}:=\prod_{\ell\in\Pi}z_\ell^{j_\ell},
\qquad \mathbf u^{\boldsymbol{\omega}}:=\prod_{K\in\Lambda}u_K^{\omega_K},
\]
and we write the valuation vector
\[
\mathbf v_\Pi(\Tam(E)):=\bigl(v_\ell(\Tam(E))\bigr)_{\ell\in\Pi}.
\]
For $X\ge 4$ and parameters with $\Reals(s)\ge 0$, we define
\begin{equation}\label{eq:ZX-def}
\mathcal Z_{\Lambda,\Pi,X}(s;u,w,\mathbf z,\mathbf u)
:=
\frac{1}{\NN(X)}
\sum_{\Ht(E)\le X}\frac{
 u^{\omega_{\Tam}(E)} w^{\Omega(\Tam(E))} \mathbf z^{\mathbf v_\Pi(\Tam(E))}\,\mathbf u^{\boldsymbol{\omega}_\Lambda(E)}}{\Tam(E)^{s}},
\end{equation}
where $\boldsymbol{\omega}_\Lambda(E)=(\omega_K(E))_{K\in\Lambda}$. The variable $u$ records $\omega_{\Tam}(E)$, the Dirichlet weight is $\Tam(E)^{-s}$, the variable $w$ records $\Omega(\Tam(E))$, the variables $z_\ell$ record the fixed $\ell$-adic valuations $v_\ell(\Tam(E))$, and each variable $u_K$ records the frequency $\omega_K(E)$ of the local reduction datum $K$.

The first main result is about the limiting behavior of this function as a multivariable Euler product. To this end, we let
\begin{equation}\label{eq:Phi-def}
\Phi_p(K;s,u,w,\mathbf z,\mathbf u):=u^{\1_{c(K)>1}}\,w^{\Omega(c(K))}\,\prod_{\ell\in\Pi}z_\ell^{v_\ell(c(K))}\,c(K)^{-s}\,\prod_{K'\in\Lambda}u_{K'}^{\1_{K=K'}}.
\end{equation}
For $K\in\KK_0$, one has $\Phi_p(K)=1$, since $c(K)=1$, $\Omega(c(K))=0$, $v_\ell(c(K))=0$ for every $\ell$, and $K\notin\Lambda$. This identity is used in the proof of absolute convergence of \eqref{eq:master-product}.

\begin{theorem}\label{thm:master}
Let $\Lambda\subseteq \KK\setminus\KK_0$ be a finite set of local reduction data, and let $\Pi$ be a finite set of primes. Suppose that $\Reals(s)\ge 0$, that $|u|\le 1$ and $|w|\le 1$, that $|z_\ell|\le 1$ for every $\ell\in\Pi$, and that $|u_K|\le 1$ for every $K\in\Lambda$. Then the limit
\[
\mathcal Z_{\Lambda,\Pi}(s;u,w,\mathbf z,\mathbf u):=
\lim_{X\to\infty}\mathcal Z_{\Lambda,\Pi,X}(s;u,w,\mathbf z,\mathbf u)
\]
exists. Moreover, the Euler product on the right side of \eqref{eq:master-product} converges absolutely and uniformly on every compact subset of the closed parameter domain
\[
\{\Reals(s)\ge 0,\ |u|\le 1,\ |w|\le 1,\ |z_\ell|\le 1\ \forall\ell\in\Pi,\ |u_K|\le 1\ \forall K\in\Lambda\},
\]
and
\begin{equation}\label{eq:master-product}
\mathcal Z_{\Lambda,\Pi}(s;u,w,\mathbf z,\mathbf u)
=
\prod_p
\left(
\beta_p+\sum_{K\in\KK\setminus\KK_0}\delta_p(K)\,\Phi_p(K;s,u,w,\mathbf z,\mathbf u)
\right).
\end{equation}
\end{theorem}

\begin{remark}
Theorem~\ref{thm:master} gives the generating function we use to study the arithmetic statistics of local reduction data and Tamagawa products for elliptic curves ordered by naive height. In particular:
\begin{enumerate}
\item Setting $\Lambda=\varnothing$, $u=w=1$, and $z_\ell=1$ for all $\ell\in\Pi$ recovers the Euler product for $L_{\Tam}(s)$;
\item Setting $\Lambda=\varnothing$ and $w=1$, $z_\ell=1$ for all $\ell\in\Pi$ gives the joint generating function for $(\omega_{\Tam}(E),\Tam(E))$;
\item Setting $\Lambda=\varnothing$, $u=1$, $s=0$, and $z_\ell=1$ encodes $\Omega(\Tam(E))$;
\item Setting $\Lambda=\varnothing$, $u=w=1$, and $s=0$ encodes the fixed valuations $\bigl(v_\ell(\Tam(E))\bigr)_{\ell\in\Pi}$;
\item Setting $\Pi=\varnothing$, $u=w=1$, and $s=0$ encodes the joint frequency vector $(\omega_K(E))_{K\in\Lambda}$ for local reduction data.
\end{enumerate}
Summing the local densities over the fibers of $c:\KK\to\Z_{\ge1}$ recovers the scalar local factors of \cite{GOT}.
\end{remark}

Whenever the relevant limiting distribution has been established, we use $\Prob$, $\E$, $\Var$, $\Cov$, and $\Corr$ for the corresponding probability, expectation, variance, covariance, and correlation.

\begin{theorem}\label{thm:omegaK}
Let $\Lambda\subseteq\KK\setminus\KK_0$ be a finite set of local reduction data. For each multi-index $\mathbf r=(r_K)_{K\in\Lambda}\in\Z_{\ge 0}^{\Lambda}$, the density
\[
\pi_\Lambda(\mathbf r):=\lim_{X\to\infty}
\frac{\#\{E:\Ht(E)\le X,\ \omega_K(E)=r_K\ \text{for all }K\in\Lambda\}}{\NN(X)}
\]
exists. Its joint generating function is
\begin{equation}\label{eq:omegaK-gf}
\sum_{\mathbf r\in\Z_{\ge 0}^{\Lambda}}\pi_\Lambda(\mathbf r)\,\mathbf u^{\mathbf r}
=
\prod_{p\ \text{prime}}\!\left(1-\sum_{K\in\Lambda}\delta_p(K)+\sum_{K\in\Lambda}\delta_p(K)u_K\right)
\qquad (\mathbf u\in\C^\Lambda).
\end{equation}
The marginal distribution of each $\omega_K(E)$ has mean and variance
\begin{equation}\label{eq:omegaK-mean-var}
\E[\omega_K(E)]=\sum_{p\ \text{prime}}\delta_p(K),\qquad
\Var(\omega_K(E))=\sum_{p\ \text{prime}}\delta_p(K)(1-\delta_p(K)),
\end{equation}
and for distinct $K,K'\in\KK\setminus\KK_0$ we have
\begin{equation}\label{eq:omegaK-cov}
\Cov(\omega_K(E),\omega_{K'}(E))=-\sum_{p\ \text{prime}}\delta_p(K)\delta_p(K').
\end{equation}
All of these sums converge absolutely.
\end{theorem}

The statistic $\omega_{\Tam}(E)$ is the aggregate of the counts $\omega_K(E)$ across those $K\in\KK\setminus\KK_0$ with $c(K)>1$. Collapsing the joint distribution of Theorem~\ref{thm:omegaK} accordingly yields the following.

\begin{theorem}[Answer to Question~\ref{q:topic1}]\label{thm:omegaE}
For each integer $r\ge 0$, the density
\[
\pi_r:=\lim_{X\to\infty}
\frac{\#\{E:\Ht(E)\le X,\ \omega_{\Tam}(E)=r\}}{\NN(X)}
\]
exists. Its generating function is
\begin{equation}\label{eq:omegaE-gf}
\sum_{r\ge 0}\pi_r u^r
=
\prod_{p\ \text{prime}}\bigl(\delta_p(1)+(1-\delta_p(1))u\bigr)
\qquad (u\in \C).
\end{equation}
The mean and variance are
\begin{equation}\label{eq:omegaE-mean-var}
\E[\omega_{\Tam}(E)] = \sum_{p\ \text{prime}} \bigl(1-\delta_p(1)\bigr)
\qquad {\text {and}}\qquad
\Var(\omega_{\Tam}(E)) = \sum_{p\ \text{prime}} \delta_p(1)\bigl(1-\delta_p(1)\bigr),
\end{equation}
and both sums converge absolutely.
\end{theorem}

As a corollary, we obtain the proportion of elliptic curves with fixed $\omega_{\Tam}(E)=r$, for every nonnegative integer $r$.

\begin{corollary}\label{cor:first-probabilities}
The following are true:
\begin{align*}
\pi_0 &= \prod_p\delta_p(1)=P_{\Tam}(1) \approx 0.50540113,\\
\pi_1 &= \sum_p\bigl(1-\delta_p(1)\bigr)\prod_{q\ne p}\delta_q(1) \approx 0.42200272,\\
\pi_2 &= \sum_{p<q}\bigl(1-\delta_p(1)\bigr)\bigl(1-\delta_q(1)\bigr)\prod_{r\ne p,q}\delta_r(1) \approx 0.06841047.
\end{align*}
More generally, we have
\[
\pi_r
=
\sum_{\substack{S\subseteq \{\text{primes}\}\\ |S|=r}}
\left(\prod_{p\in S}(1-\delta_p(1))\right)
\left(\prod_{q\notin S}\delta_q(1)\right).
\]
\end{corollary}

Question~\ref{q:topic2} is answered by the specialization $\Lambda=\varnothing$, $\Pi=\varnothing$, and $w=1$ in Theorem~\ref{thm:master}. Namely, we obtain the following formula.

\begin{corollary}[Answer to Question~\ref{q:topic2}]\label{cor:joint-omega-tam}
For $\Reals(s)\ge 0$ and $|u|\le 1$ we have
\[
\lim_{X\to\infty}\frac{1}{\NN(X)}\sum_{\Ht(E)\le X}u^{\omega_{\Tam}(E)}\Tam(E)^{-s}
=
\prod_p\left(\delta_p(1)+u\sum_{t\ge 2}\delta_p(t)t^{-s}\right).
\]
\end{corollary}

We turn to Question~\ref{q:topic3}. For the additive statistic $\Omega(\Tam(E))$, we have the following result.

\begin{theorem}\label{thm:Omega}
For each integer $b\ge 0$, the density
\[
\rho_b:=\lim_{X\to\infty}
\frac{\#\{E:\Ht(E)\le X,\ \Omega(\Tam(E))=b\}}{\NN(X)}
\]
exists. Its generating function is
\begin{equation}\label{eq:Omega-gf}
\sum_{b\ge 0}\rho_b w^b
=
\prod_{p\ \text{prime}}\left(\sum_{t\ge 1}\delta_p(t)w^{\Omega(t)}\right)
\qquad (|w|\le 1).
\end{equation}
Moreover, we have
\begin{equation}\label{eq:Omega-mean}
\E[\Omega(\Tam(E))]
=
\sum_p\sum_{t\ge 1}\delta_p(t)\Omega(t)
<\infty.
\end{equation}
Consequently, we have
\begin{equation}\label{eq:smallomega-bound}
\E[\omega(\Tam(E))]\le \E[\Omega(\Tam(E))]<\infty.
\end{equation}
\end{theorem}

The next theorem gives the Euler-product formula used for the nonadditive statistic $\omega(\Tam(E))$.

\begin{theorem}\label{thm:valuations}
Let $\Pi$ be a finite set of primes. For each multi-index $\mathbf j=(j_\ell)_{\ell\in\Pi}\in\Z_{\ge 0}^{\Pi}$, the density
\[
Q_\Pi(\mathbf j)
:=
\lim_{X\to\infty}
\frac{\#\{E:\Ht(E)\le X,\ v_\ell(\Tam(E))=j_\ell\ \text{for all }\ell\in\Pi\}}{\NN(X)}
\]
exists, and the generating function is
\begin{equation}\label{eq:valuation-gf}
\sum_{\mathbf j\in\Z_{\ge 0}^{\Pi}}Q_\Pi(\mathbf j)\mathbf z^{\mathbf j}
=
\prod_{p\ \text{prime}}\left(\sum_{t\ge 1}\delta_p(t)\prod_{\ell\in\Pi}z_\ell^{v_\ell(t)}\right)
\qquad (|z_\ell|\le 1).
\end{equation}
In particular, for any finite set $A$ of primes, we have
\begin{equation}\label{eq:noA-divisibility}
\Prob\bigl(\ell\nmid \Tam(E)\ \text{for all }\ell\in A\bigr)
=
\prod_{p\ \text{prime}}\left(\sum_{\substack{t\ge 1 \\ (t,\prod_{\ell\in A}\ell)=1}}\delta_p(t)\right).
\end{equation}
Hence, for a fixed prime $\ell$, we have
\begin{equation}\label{eq:ell-divides-Tam}
\Prob\bigl(\ell\mid \Tam(E)\bigr)
=
1-
\prod_p\left(\sum_{\ell\nmid t}\delta_p(t)\right).
\end{equation}
In particular, the density of curves with odd Tamagawa product is
\begin{equation}\label{eq:odd-density}
\Prob\bigl(\Tam(E)\text{ is odd}\bigr)
=
\prod_p\left(\sum_{t\text{ odd}}\delta_p(t)\right).
\end{equation}
\end{theorem}

In terms of $\ell$-adic valuations, Theorem~\ref{thm:valuations} gives the following consequence.

\begin{corollary}\label{cor:ell-adic-mean}
For any fixed prime $\ell$, we have
\[
\E[v_\ell(\Tam(E))]
=
\sum_{p\ \text{prime}}\sum_{t\ge 1}\delta_p(t)v_\ell(t)
<\infty.
\]
\end{corollary}

The next theorem records the exact covariance structure among $\ell$-adic valuations of $\Tam(E)$.

\begin{theorem}\label{thm:covariance-vl}
Suppose that $\ell$ and $\ell'$ are distinct primes. Then the following are true.

\noindent
(1) We have
\begin{displaymath}
\begin{split}
\Cov(v_\ell(\Tam(E)),&v_{\ell'}(\Tam(E)))
\\
&
\ \ \ =\sum_p\Bigl(
\sum_{t\ge 1}\delta_p(t)v_\ell(t)v_{\ell'}(t)
-
\bigl(\textstyle\sum_{t\ge 1}\delta_p(t)v_\ell(t)\bigr)
\bigl(\textstyle\sum_{t\ge 1}\delta_p(t)v_{\ell'}(t)\bigr)
\Bigr),
\end{split}
\end{displaymath}
where the outer sum converges absolutely.

\noindent
(2) This covariance is strictly negative; for example
\[
\Cov(v_2(\Tam(E)),v_3(\Tam(E)))\approx -0.02489.
\]

 \noindent
 (3) There is an absolute constant $C>0$ such that 
$$
|\Cov(v_\ell(\Tam(E)),v_{\ell'}(\Tam(E)))|\le \frac{C}{(\ell\ell')^{2}}.
$$
In particular, the covariance tends to $0$ as $\ell\ell'\to\infty$.
\end{theorem}

We also record the exact linear correlation between $\omega_{\Tam}(E)$ and $\Omega(\Tam(E))$.

\begin{theorem}\label{thm:corr}
We have
\begin{equation}\label{eq:cov-omega-Omega}
\Cov\bigl(\omega_{\Tam}(E),\Omega(\Tam(E))\bigr)
=
\sum_p\delta_p(1)\sum_{t\ge 1}\delta_p(t)\Omega(t),
\end{equation}
where the outer sum converges absolutely. Consequently,
\begin{equation}\label{eq:corr-omega-Omega}
\Corr\bigl(\omega_{\Tam}(E),\Omega(\Tam(E))\bigr)
=
\frac{\sum_p \delta_p(1)\sum_{t\ge 1}\delta_p(t)\Omega(t)}{\sqrt{\Var(\omega_{\Tam}(E))\,\Var(\Omega(\Tam(E)))}}
\approx 0.9801.
\end{equation}
\end{theorem}

Finally, we address Question~\ref{q:topic4}.

\begin{theorem}[Answer to Question~\ref{q:topic4}]\label{thm:tail}
For every integer $k\ge 0$, the $k$th moment
\begin{equation}\label{eq:moment-def}
M_k:=\sum_{m\ge 1}P_{\Tam}(m)m^k
\end{equation}
converges absolutely and is given by the Euler product
\begin{equation}\label{eq:moment-product}
M_k=
\prod_p\left(\sum_{t\ge 1}\delta_p(t)t^k\right).
\end{equation}
Consequently, for every integer $k\ge 1$ and every $M\ge 1$, we have
\begin{equation}\label{eq:tail-bound}
\sum_{m\ge M}P_{\Tam}(m)
\le
\frac{M_k}{M^k}.
\end{equation}
In particular, the cumulative tail $\sum_{m\ge M}P_{\Tam}(m)$ decays faster than every fixed negative power of $M$.
\end{theorem}

\subsection{Examples and numerics}
\label{subsec:numerics}
We record numerical illustrations of the main theorems. To assess cutoff stability, we computed the relevant Euler products at the prime cutoffs $p\le 10^6$, $p\le 10^7$, and $p\le 10^8$. For each prime $p$ in the cutoff we used the exact rational formulas in Theorem~\ref{thm:GOT-local}; the rows for expectations, variances, covariances, and correlations were then obtained by the corresponding convergent sums displayed in Theorems~\ref{thm:omegaE}, \ref{thm:Omega}, \ref{thm:covariance-vl}, and \ref{thm:corr}. The omitted tail is bounded by the local estimates in Lemma~\ref{lem:local-estimates}: for every scalar Euler product or first moment the tail after $Y$ is $O(\sum_{p>Y}p^{-2})$, and for $\Cov(v_2,v_3)$ the proof of Theorem~\ref{thm:covariance-vl} gives the sharper tail $O(\sum_{p>Y}p^{-5})$. Thus the table is reproducible directly from the displayed formulas; the numerical values quoted in the text are the $p\le 10^8$ truncations.

\begin{table}[ht]
\centering
\small
\begin{tabular}{lccc}
\hline
Statistic & $p\le 10^6$ & $p\le 10^7$ & $p\le 10^8$ \\
\hline
$\pi_0$ & 0.505401163567 & 0.505401132275 & 0.505401129574 \\
$\pi_1$ & 0.422002711220 & 0.422002716384 & 0.422002716829 \\
$\pi_2$ & 0.068410442333 & 0.068410464226 & 0.068410466116 \\
$\E[\omega_{\Tam}(E)]$ & 0.571499207875 & 0.571499269791 & 0.571499275135 \\
$\Var(\omega_{\Tam}(E))$ & 0.407537905234 & 0.407537967150 & 0.407537972493 \\
$\E[\Omega(\Tam(E))]$ & 0.589821573544 & 0.589821635460 & 0.589821640804 \\
$\Prob(\Tam(E)\text{ is odd})$ & 0.575088754178 & 0.575088718571 & 0.575088715498 \\
$\E[\Tam(E)]$ & 1.818615749367 & 1.818615861968 & 1.818615871687 \\
$\Corr(\omega_{\Tam}(E),\Omega(\Tam(E)))$ & {0.980088084} & {0.980088087} & {0.980088087} \\
{$\Cov(v_2(\Tam(E)),v_3(\Tam(E)))$} & {$-0.024890$} & {$-0.024890$} & {$-0.024890$} \\
\hline
\end{tabular}
\caption{Numerical stability of the truncated Euler products.}
\label{tab:cutoff-stability}
\end{table}

\subsubsection{The distribution of \texorpdfstring{$\omega_{\Tam}(E)$}{the Tamagawa prime count}}

By Corollary~\ref{cor:first-probabilities} and Theorem~\ref{thm:omegaE}, the first few probabilities are given by explicit convergent Euler products. Numerically, one finds
\[
\pi_0 = \Prob(\omega_{\Tam}(E)=0) = P_{\Tam}(1) \approx 0.50540113,
\]
\[
\pi_1 = \Prob(\omega_{\Tam}(E)=1) \approx 0.42200272,
\qquad {\text {\rm and}}\qquad
\pi_2 = \Prob(\omega_{\Tam}(E)=2) \approx 0.06841047.
\]
Thus more than half of all curves in this height ordering have trivial Tamagawa product, about $42.2\%$ have exactly one prime with nontrivial local Tamagawa number, and only about $6.8\%$ have exactly two such primes.

The mean and variance from Theorem~\ref{thm:omegaE} are numerically
\[
\E[\omega_{\Tam}(E)] \approx 0.57149928
\qquad {\text {\rm and}}\qquad
\Var(\omega_{\Tam}(E)) \approx 0.40753797.
\]
So the number of primes contributing nontrivially to $\Tam(E)$ is typically very small.

\subsubsection{The additive arithmetic of \texorpdfstring{$\Tam(E)$}{Tam(E)}}

Theorem~\ref{thm:Omega} gives an Euler product for this additive statistic. Numerically, the resulting mean is
\[
\E[\Omega(\Tam(E))] \approx 0.58982164.
\]
This is only slightly larger than $\E[\omega_{\Tam}(E)]$. The correlation between the two statistics is approximately $0.9801$ by Theorem~\ref{thm:corr}.

From Theorem~\ref{thm:valuations}, the density of curves with odd Tamagawa product is
\[
\Prob\bigl(\Tam(E)\text{ is odd}\bigr)
=
\prod_p\Bigl(\sum_{t\text{ odd}}\delta_p(t)\Bigr)
\approx 0.57508872.
\]
Equivalently, we have
\[
\Prob\bigl(2\mid \Tam(E)\bigr) \approx 0.42491128.
\]

Theorem~\ref{thm:covariance-vl} records the exact covariance between different $\ell$-adic valuations of $\Tam(E)$. At $\ell=2$, $\ell'=3$, the numerical value is
\[
\Cov(v_2(\Tam(E)),v_3(\Tam(E)))\approx -0.02489.
\]
The sign is negative for every pair of distinct primes, and Theorem~\ref{thm:covariance-vl}(3) gives the decay as $\ell\ell'\to\infty$.

\subsubsection{A moment example}

The case $k=1$ of Theorem~\ref{thm:tail} gives the mean of the limiting distribution of $\Tam(E)$
\[
M_1 = \E[\Tam(E)] = \prod_p\left(\sum_{t\ge 1}\delta_p(t)t\right).
\]
Truncating this Euler product at $p\le 10^8$ gives
\[
\E[\Tam(E)] \approx 1.81861587.
\]
In particular, this corrects the value $1.8193\dots$ quoted in Theorem~1.5 of \cite{GOT}.

\begin{remark}
Theorem~\ref{thm:master} gives the Euler products for $\omega_{\Tam}(E)$, $\Omega(\Tam(E))$, and the fixed valuation vector $\mathbf v_\Pi(\Tam(E))$. The statistic $\omega(\Tam(E))$, the number of distinct prime divisors of the global integer $\Tam(E)$, is not additive in the local factors $c_p(E)$. As a consequence, there is no one-variable Euler product obtained by inserting a factor $u^{\omega(t)}$ at each prime $p$. We therefore study fixed valuation statistics first, and then recover finite-support divisibility information from them. This gives \eqref{eq:noA-divisibility}--\eqref{eq:ell-divides-Tam}.
\end{remark}

The paper is organized as follows. Section~\ref{sec:local_densities} collects the local densities $\delta_p(K)$ for local reduction data, proves the uniform local estimates needed for absolute convergence, and establishes the equidistribution, large-prime tail, and finite-prime factorization lemmas that reduce the global questions to finitely many local conditions. Section~\ref{sec:proofs} proves the master Euler product. Section~\ref{sec:analytic} records the analytic coefficient-extraction lemmas. Finally, Section~\ref{sec:distribution-proofs} derives the limiting distributions of the counts $\omega_K(E)$ for local reduction data, $\omega_{\Tam}(E)$ and $\Omega(\Tam(E))$, the valuation statistics for $\Tam(E)$, the covariance and correlation theorems, and the moment and tail bounds for the Tamagawa product. In Section~\ref{sec:AxiomProver}, we describe the autoformalization of the results of this paper in Lean by AxiomProver.

\section{Local densities}\label{sec:local_densities}

In this section we collect the local facts needed later. The explicit densities come from \cite{GOT}, with the two corrections of Remarks~\ref{rem:local-normalization} and~\ref{rem:three-adic-correction}; the other results are proved here, at the level of local reduction data. The discussion is based on Tate's algorithm, which computes the Kodaira symbol, the splitting data, the component group, and the local Tamagawa number at all primes $p$. For background, see Chapter~IV of \cite{SilvermanAdvanced} or \cite{Tate}.

\subsection{The local densities}
We define the local densities $\delta_p(K)$ and $\delta_p(t)$, and state the explicit formulas for $\delta_p(t)$ on which all of our numerical values rest, together with the two corrections to \cite[Lemma~3.1]{GOT} that they require.

For each prime $p$ and local reduction datum $K\in\KK$, let $\delta_p(K)$ denote the proportion of short Weierstrass models whose $p$-minimal model has datum $K$. Equivalently, $\delta_p(K)$ is the normalized Haar measure of the corresponding subset of $\Z_p^2$; the singular locus $\Delta=0$ has Haar measure zero, so this agrees with the proportion language used in \cite{GOT}. These local densities satisfy
\begin{equation}\label{eq:delta-normalization}
\sum_{K\in\KK}\delta_p(K)=1,
\end{equation}
and the scalar densities of \cite{GOT} are obtained by aggregating over local reduction data of fixed local Tamagawa number (see \eqref{deltapt}):
\begin{equation}\label{eq:delta-collapse}
\delta_p(t)=\sum_{K\in\KK\,:\,c(K)=t}\delta_p(K)\qquad (t\ge 1).
\end{equation}
The densities $\delta_p(K)$ are those obtained from the cases of Tate's algorithm, as tabulated in Table~A1 of \cite{GOT} for $p\ge5$ and in Tables~A2 and~A3 for $p=2,3$, with the correction to Table~A3 described in Remark~\ref{rem:three-adic-correction}; in the present paper we only need their normalization, their aggregation to the scalar densities, and the uniform estimate \(\sum_{K\notin\KK_0}\delta_p(K)\ll p^{-2}\). Griffin, Tsai, and one of the authors \cite[\S 2]{GOT} computed the scalar densities $\delta_p(t)$ explicitly, as follows.

\begin{theorem}\label{thm:GOT-local}
The numbers $\delta_p(t)$ satisfy the following formulas.
\begin{enumerate}
\item For $p=2$, we have
\[
\delta_2(1)=\frac{241}{396},\qquad
\delta_2(2)=\frac{7495}{24552},\qquad
\delta_2(3)=\frac{1153}{16368},\qquad
\delta_2(4)=\frac{171}{10912},
\]
and for $t\ge 5$, we have
\[
\delta_2(t)=\frac{1}{2^{t+1}\cdot 1023}.
\]

\item For $p=3$, we have
\[
\delta_3(1)=\frac{1924841}{2125728},\qquad
\delta_3(2)=\frac{509345}{6377184},\qquad
\delta_3(3)=\frac{30619}{2391444},\qquad
\delta_3(4)=\frac{1193}{652212},
\]
and for $t\ge 5$, we have
\[
\delta_3(t)=\frac{1}{3^{t+1}\cdot 29524}.
\]

\item If $p\ge 5$, then we have
\[
\delta_p(t)=
\begin{cases}
\displaystyle
1-
\frac{p(6p^7+9p^6+9p^5+7p^4+8p^3+7p^2+9p+6)}{6(p+1)^2(p^8+p^6+p^4+p^2+1)}, & t=1,\\[14pt]
\displaystyle
\frac{p(2p^7+2p^6+p^5+p^4+2p^3+p^2+2p+2)}{2(p+1)^2(p^8+p^6+p^4+p^2+1)}, & t=2,\\[14pt]
\displaystyle
\frac{p^2(p^4+1)}{2(p+1)(p^8+p^6+p^4+p^2+1)}, & t=3,\\[14pt]
\displaystyle
\frac{p^3(3p^2-2p+1)}{6(p+1)(p^8+p^6+p^4+p^2+1)}, & t=4,\\[14pt]
\displaystyle
\frac{p^{10}-2p^9+p^8}{2p^t(p^{10}-1)}, & t\ge 5.
\end{cases}
\]
\end{enumerate}
Moreover, we have
\begin{equation}\label{eq:GOT-Euler}
\sum_{m\ge 1}\frac{P_{\Tam}(m)}{m^s}
=
\prod_p\left(\sum_{t\ge 1}\delta_p(t)t^{-s}\right)
\end{equation}
for every $s$ for which the Euler product converges absolutely.
\end{theorem}

\begin{proof}
The explicit formulas are Lemma~3.1 of \cite{GOT}, aggregated via \eqref{eq:delta-collapse} from the densities for local reduction data in Tables~A1--A3 of \cite{GOT}. Identity \eqref{eq:GOT-Euler} is Theorem~1.1 of \cite{GOT}. For the case $p\ge 5$ and $t=4$, we use the corrected numerator $3p^2-2p+1$, as explained in Remark~\ref{rem:local-normalization}. For the case $p=3$ and $t\in\{1,2,3\}$, we use the corrected values explained in Remark~\ref{rem:three-adic-correction}.\qedhere
\end{proof}

\begin{remark}\label{rem:local-normalization}
For each prime $p$, the local densities satisfy
\[
\sum_{t\ge 1}\delta_p(t)=1.
\]
For $p=2,3$ this is immediate from the explicit geometric tails. For $p\ge 5$, this is a direct algebraic check from the formulas in Theorem~\ref{thm:GOT-local}; with the $t=4$ numerator written as $3p^2-2p+1$, multiplying the sum of the five displayed cases by the common denominator $6(p+1)^2(p^{10}-1)$ gives exactly $6(p+1)^2(p^{10}-1)$. Equivalently, the $t=4$ term is the sum of the $I_4$, $I_0^\ast$, and $I_n^\ast$ contributions in Table~5 of \cite{GOT}, and that sum simplifies to the displayed numerator $3p^2-2p+1$. In particular, the numerator $3p^2-2p-1$ printed in Lemma~3.1 of \cite{GOT} for this case is an error. The corresponding values in Table~\ref{tab:cutoff-stability} use this corrected local factor.
\end{remark}

\begin{remark}\label{rem:three-adic-correction}
The values of $\delta_3(1)$, $\delta_3(2)$, and $\delta_3(3)$ in Theorem~\ref{thm:GOT-local}(2) differ from those printed in Lemma~3.1(2) of \cite{GOT}, namely
\[
\frac{1924625}{2125728},\qquad \frac{510641}{6377184},\qquad \frac{7594}{597861},
\]
by $3/29524$, $-6/29524$, and $3/29524$ respectively. The discrepancy originates in Case~9(iii) of the analysis at $p=3$ in \cite[\S2.2]{GOT}, where the family
\[
a_4=-27t^2,\qquad a_6=54t^3+81s,\qquad 3\nmid st,
\]
with $v_3(a_4)=v_3(a_6)=3$ and $v_3(\Delta)=10$, is assigned Kodaira type $III^\ast$ with $c_3=2$. The translation $x\mapsto x+3t$ transforms these models into $y^2=x^3+9tx^2+81s$. In Tate's algorithm, the auxiliary cubic then has a triple root, and the auxiliary quadratic $Y^2-s\pmod 3$ has distinct roots since $s$ is a unit, so the type is $IV^\ast$, with $c_3=3$ if $s\equiv 1\pmod 3$ and $c_3=1$ if $s\equiv -1\pmod 3$. For instance, $y^2=x^3-27x+135$ has type $IV^\ast$ at $3$ with $c_3=3$, and $y^2=x^3-27x-27$ has type $IV^\ast$ with $c_3=1$. The affected coefficient classes are $a_4\equiv 54\pmod{81}$ and $a_6\equiv 27,108,135,216\pmod{243}$, of total Haar measure $4\cdot 3^{-9}$; the resulting change $(2,-4,2)\cdot 3^{-9}$ in the densities of $3$-minimal models becomes $(q,-2q,q)$ with $q=3/29524$ after the factor $(1-3^{-10})^{-1}$ accounting for non-minimal models. Since the shift sums to zero and preserves the first moment $\sum_{t\ge 1}t\,\delta_3(t)=262401/236192$, it is invisible both to the normalization check of Remark~\ref{rem:local-normalization} and to the value of $\E[\Tam(E)]$.
\end{remark}

\subsection{The \texorpdfstring{$\KK_0$}{K0} asymptotic}
Almost all of the mass of $\delta_p$ lies on the three data in $\KK_0$, none of which affects any of our statistics. This subsection shows that the remaining mass is $O(p^{-2})$.

The three local reduction data in
\[
\KK_0=\{I_0,\,I_1^{\mathrm{sp}},\,I_1^{\mathrm{ns}}\}
\]
all satisfy $c(K)=1$, $\Omega(c(K))=0$, and $v_\ell(c(K))=0$ for every $\ell$. The following asymptotic, which refines the scalar estimate $1-\delta_p(1)\ll p^{-2}$ of \cite{GOT}, is the mechanism that allows the master Euler product of Theorem~\ref{thm:master} to converge.

\begin{lemma}\label{lem:K0-asymptotic}
For every prime $p\ge 5$, we have
\begin{equation}\label{eq:K0-mass}
1-\beta_p=1-\delta_p(I_0)-\delta_p(I_1^{\mathrm{sp}})-\delta_p(I_1^{\mathrm{ns}})=\sum_{K\notin\KK_0}\delta_p(K)\ll \frac{1}{p^2}.
\end{equation}
Moreover, the individual $\KK_0$-densities satisfy
\begin{equation}\label{eq:K0-individual}
\delta_p(I_0)=1-O\!\left(\frac{1}{p}\right),\qquad
\delta_p(I_1^{\mathrm{sp}})\asymp \delta_p(I_1^{\mathrm{ns}})\asymp \frac{1}{p}.
\end{equation}
\end{lemma}

\begin{proof}
For every $n\ge 1$ and $p\ge 5$, the split and nonsplit multiplicative densities are
\[
\delta_p(I_n^{\mathrm{sp}})=\delta_p(I_n^{\mathrm{ns}})=\frac{(p-1)^2}{2p^{n+2}}\bigl(1-p^{-10}\bigr)^{-1}.
\]
Indeed, the nodal locus modulo $p$ consists of the $p-1$ residue pairs $(a_4,a_6)\equiv(-3u^2,2u^3)$ with $u\in\mathbb{F}_p^{\times}$. On each of them, the condition $v_p(\Delta)=n$ has relative density $(p-1)/p^{n}$, since $\partial\Delta/\partial a_6=-864a_6$ is a unit there. The quadratic twist $(a_4,a_6)\mapsto(\nu^2a_4,\nu^3a_6)$ by a quadratic non-residue $\nu$ is a measure-preserving bijection that preserves $v_p(\Delta)$ and exchanges the split and nonsplit branches. The factor $(1-p^{-10})^{-1}$ accounts for the non-minimal short models. The additive types contribute a total density $O(p^{-2})$. Summing the multiplicative densities over $n\ge 1$ gives the geometric series
\[
\sum_{n\ge 1}(\delta_p(I_n^{\mathrm{sp}})+\delta_p(I_n^{\mathrm{ns}}))
=
\frac{p-1}{p^2}\bigl(1-p^{-10}\bigr)^{-1}
=
\frac{1}{p}-\frac{1}{p^2}+O(p^{-11}),
\]
so the total bad-reduction density is $1/p+O(p^{-2})$, and \eqref{eq:K0-individual} follows. The $\KK_0$-mass estimate \eqref{eq:K0-mass} then comes from
\[
\sum_{K\notin\KK_0}\delta_p(K)
=
\sum_{n\ge 2}\bigl(\delta_p(I_n^{\mathrm{sp}})+\delta_p(I_n^{\mathrm{ns}})\bigr)+\delta_p(\text{additive types})
=O(p^{-2}),
\]
since the tail of the geometric series is $O(p^{-2})$ and the additive contribution is $O(p^{-2})$.
\end{proof}

\subsection{Local estimates}

The previous formulas imply the soft estimates that drive all of our convergence arguments.

\begin{lemma}\label{lem:local-estimates}
For primes $p\ge 5$ we have
\begin{equation}\label{eq:d1-estimate}
1-\delta_p(1)\ll \frac{1}{p^2}.
\end{equation}
More generally, for every fixed integer $k\ge 0$, we have
\begin{equation}\label{eq:moment-local-estimate}
\sum_{t\ge 1}\delta_p(t)t^k = 1+O_k\!\left(\frac{1}{p^2}\right)
\qquad (p\ge 5).
\end{equation}
Finally, let $\Lambda\subseteq\KK\setminus\KK_0$ be finite, let $\Pi$ be a finite set of primes, and suppose that $|u|\le 1$, $|w|\le 1$, $|z_\ell|\le 1$ for every $\ell\in\Pi$, and $|u_K|\le 1$ for every $K\in\Lambda$. Then we have
\begin{equation}\label{eq:bounded-local-estimate}
\beta_p+\sum_{K\in\KK\setminus\KK_0}\delta_p(K)\,\Phi_p(K;s,u,w,\mathbf z,\mathbf u)
=
1+O\!\left(\frac{1}{p^2}\right)
\end{equation}
for every $s$ with $\Reals(s)\ge 0$. Moreover, for every fixed prime $\ell$, we have
\begin{equation}\label{eq:Omega-local-estimate}
\sum_{t\ge 1}\delta_p(t)\Omega(t)\ll \frac{1}{p^2},
\qquad
\sum_{t\ge 1}\delta_p(t)v_\ell(t)\ll \frac{1}{p^2}
\qquad (p\ge 5),
\end{equation}
\begin{equation}\label{eq:second-moment-local-estimate}
\sum_{t\ge 1}\delta_p(t)\Omega(t)^2\ll \frac{1}{p^2},
\qquad
\sum_{t\ge 1}\delta_p(t)v_\ell(t)v_{\ell'}(t)\ll \frac{1}{p^2}
\qquad (p\ge 5),
\end{equation}
and for every fixed $K\in\KK\setminus\KK_0$,
\begin{equation}\label{eq:K-local-estimate}
\delta_p(K)\ll \frac{1}{p^2}
\qquad (p\ge 5).
\end{equation}
\end{lemma}

\begin{proof}
For \eqref{eq:d1-estimate}, inspect the explicit formula for $\delta_p(1)$ in Theorem~\ref{thm:GOT-local}. The numerator has degree $8$ in $p$ and the denominator has degree $10$, so the difference $1-\delta_p(1)$ is $O(p^{-2})$.

For \eqref{eq:moment-local-estimate}, split the sum into $t=1,2,3,4$ and $t\ge 5$. The $t=1$ term equals $1+O(p^{-2})$, while the terms with $t=2,3,4$ are each $O_k(p^{-2})$ because the corresponding formulas are visibly $O(p^{-2})$. For the tail, Theorem~\ref{thm:GOT-local} gives
\[
\delta_p(t)=\frac{p^8(p-1)^2}{2(p^{10}-1)}\cdot p^{-t}
\ll p^{-t}
\qquad (t\ge 5),
\]
so
\[
\sum_{t\ge 5}\delta_p(t)t^k
\ll
\sum_{t\ge 5}\frac{t^k}{p^t}
\ll_k \frac{1}{p^5}.
\]
This proves \eqref{eq:moment-local-estimate}.

For \eqref{eq:bounded-local-estimate}, recall from \eqref{eq:Phi-def} that $\Phi_p(K)=1$ for every $K\in\KK_0$, so the $\KK_0$ contribution is exactly $\beta_p$ and cancels the leading $1$ up to $1-\beta_p\ll p^{-2}$ by Lemma~\ref{lem:K0-asymptotic}. The hypotheses on the parameters give $|\Phi_p(K)|\le 1$ for every $K$, so
\[
\left|\beta_p+\sum_{K\notin\KK_0}\delta_p(K)\Phi_p(K)-1\right|
=
\left|\sum_{K\notin\KK_0}\delta_p(K)\bigl(\Phi_p(K)-1\bigr)\right|
\le
2(1-\beta_p)
\ll
\frac{1}{p^2}.
\]
This proves \eqref{eq:bounded-local-estimate}.

For \eqref{eq:Omega-local-estimate}, note that $\Omega(1)=v_\ell(1)=0$ and, for every $t\ge 2$,
\[
0\le v_\ell(t)\le \Omega(t)\le t.
\]
Hence
\[
\sum_{t\ge 1}\delta_p(t)\Omega(t)
\le
\sum_{t\ge 2}\delta_p(t)t
=
\left(\sum_{t\ge 1}\delta_p(t)t\right)-\delta_p(1)
=O(p^{-2})
\]
by \eqref{eq:moment-local-estimate} with $k=1$, and the same bound follows for $\sum_t\delta_p(t)v_\ell(t)$.

For \eqref{eq:second-moment-local-estimate}, use $v_\ell(t)v_{\ell'}(t)\le \Omega(t)^2\le t^2$ for $t\ge1$ and $\Omega(1)=0$. Thus
\[
\sum_{t\ge1}\delta_p(t)\Omega(t)^2
\le
\sum_{t\ge2}\delta_p(t)t^2
=
\sum_{t\ge1}\delta_p(t)t^2-\delta_p(1)
=O(p^{-2})
\]
by \eqref{eq:moment-local-estimate} with $k=2$, together with \eqref{eq:d1-estimate}. The mixed valuation bound follows immediately.

Finally, \eqref{eq:K-local-estimate} is immediate from Lemma~\ref{lem:K0-asymptotic}, since $\delta_p(K)\le \sum_{K'\notin\KK_0}\delta_p(K')\ll p^{-2}$.
\end{proof}

\subsection{Equidistribution in residue classes}

We shall need the elementary fact that the height boxes become equidistributed modulo every fixed modulus.

\begin{lemma}\label{lem:equidistribution}
Let $M\ge 1$ be fixed. For every residue class $(r,s)\in (\Z/M\Z)^2$, we have
\[
\lim_{X\to\infty}
\frac{1}{\NN(X)}
\#\Bigl\{(a_4,a_6)\in\Z^2:
\begin{array}{l}
\Ht(E(a_4,a_6))\le X,\ \Delta(E)\neq 0,\\
(a_4,a_6)\equiv (r,s)\!\!\pmod M
\end{array}
\Bigr\}
=
\frac{1}{M^2}.
\]
More generally, if $A\subseteq (\Z/M\Z)^2$, then the density of curves whose coefficient pair lies in $A$ modulo $M$ exists and equals $|A|/M^2$.
\end{lemma}

\begin{proof}
We let
\[
A_X:=\left\lfloor\left(\frac{X}{4}\right)^{1/3}\right\rfloor
\qquad \text{and}\qquad
B_X:=\left\lfloor\left(\frac{X}{27}\right)^{1/2}\right\rfloor.
\]
Then the condition $\Ht(E(a_4,a_6))\le X$ is equivalent to $|a_4|\le A_X$ and $|a_6|\le B_X$. Hence the total number of integer pairs in the box is
\[
(2A_X+1)(2B_X+1)=\frac{4}{4^{1/3}27^{1/2}}X^{5/6}+O(X^{1/2}+X^{1/3}).
\]
The singular pairs satisfy $4a_4^3+27a_6^2=0$. For each $a_6$ there is at most one possible value of $a_4$, so the number of singular pairs in the box is $O(B_X)=O(X^{1/2})$. Therefore
\[
\NN(X)=(2A_X+1)(2B_X+1)+O(X^{1/2})
=\frac{4}{4^{1/3}27^{1/2}}X^{5/6}+o(X^{5/6}).
\]
Now fix $(r,s)\bmod M$. The number of pairs in the box satisfying these congruence conditions is
\[
\frac{2A_X+1}{M}\cdot \frac{2B_X+1}{M}+O(A_X+B_X+1)
=
\frac{(2A_X+1)(2B_X+1)}{M^2}+o(X^{5/6}),
\]
and again removing the singular pairs changes the count by only $O(X^{1/2})=o(X^{5/6})$. Dividing by $\NN(X)$ gives the claim for a single residue class, and summing over residue classes in $A$ proves the general statement.
\end{proof}

\subsection{A uniform large-prime tail estimate}

To pass from finitely many local conditions to the full Euler product, we need a tail estimate that is uniform in $X$.

\begin{lemma}\label{lem:uniform-large-prime-tail}
For every prime $p\ge 5$ and every $X\ge 4$, we have
\[
\frac{1}{\NN(X)}
\#\{E(a_4,a_6): \Ht(E)\le X,\ \tau_p(E)\notin\KK_0\}
\ll
\frac{1}{p^2}+\frac{1}{pX^{1/3}}+\frac{1}{X^{1/2}}.
\]
Consequently, for every real number $Y\ge 5$, we have
\[
\limsup_{X\to\infty}
\frac{1}{\NN(X)}
\sum_{\Ht(E)\le X}\sum_{p>Y}\1_{\tau_p(E)\notin\KK_0}
\ll
\sum_{p>Y}\frac{1}{p^2},
\]
and the right-hand side tends to $0$ as $Y\to\infty$.
\end{lemma}

\begin{proof}
If $4\le X<27$, the only nonsingular models in the height box are $y^2=x^3\pm x$, whose discriminants are $\mp64$. Both have good reduction at every prime $p\ge5$, so the proportion in the first assertion is zero. We may therefore assume that $X\ge27$.

Write
\[
A_X:=\left\lfloor\left(\frac{X}{4}\right)^{1/3}\right\rfloor
\qquad {\text {\rm and}}\qquad
B_X:=\left\lfloor\left(\frac{X}{27}\right)^{1/2}\right\rfloor,
\]
and let
\[
N_p(X):=\#\{E(a_4,a_6): \Ht(E)\le X,\ \tau_p(E)\notin\KK_0\}.
\]
If $p\ge 5$ and $\tau_p(E)\notin\KK_0$, then Tate's algorithm shows that the $p$-minimal discriminant has valuation at least $2$. Since the discriminant of the original integral short Weierstrass equation differs from the minimal discriminant by a factor with $p$-adic valuation divisible by $12$, this implies $v_p(\Delta(E))\ge 2$. Indeed, if the minimal discriminant valuation is at most $1$, then $\tau_p(E)$ is $I_0$, $I_1^{\mathrm{sp}}$ or $I_1^{\mathrm{ns}}$, that is, $\tau_p(E)\in\KK_0$. Thus, every curve counted by $N_p(X)$ satisfies
\[
p^2\mid \Delta(E)=-16(4a_4^3+27a_6^2),
\]
so it is enough to count pairs $(a_4,a_6)$ in the height box for which
\[
4a_4^3+27a_6^2\equiv 0 \pmod{p^2}.
\]

Fix $a_4$ with $|a_4|\le A_X$. If $p\nmid a_4$, then
\[
27a_6^2\equiv -4a_4^3 \pmod{p^2}
\]
has at most two solutions for $a_6\bmod p^2$, because a nonzero residue modulo $p^2$ has at most two square roots. Hence, the number of such integers $a_6$ with $|a_6|\le B_X$ is
\[
\ll \frac{B_X}{p^2}+1.
\]
If instead $p\mid a_4$, then the same congruence forces $p\mid a_6$, so the number of admissible integers $a_6$ is
\[
\ll \frac{B_X}{p}+1.
\]
Since the number of integers $a_4$ with $|a_4|\le A_X$ and $p\mid a_4$ is $\ll A_X/p+1$, it follows that
\[
N_p(X)
\ll
(A_X+1)\left(\frac{B_X}{p^2}+1\right)
+
\left(\frac{A_X}{p}+1\right)\left(\frac{B_X}{p}+1\right).
\]
Therefore, we find that
\[
N_p(X)
\ll
\frac{A_XB_X}{p^2}+A_X+\frac{B_X}{p}+1.
\]
By Lemma~\ref{lem:equidistribution} we have $\NN(X)\asymp A_XB_X$, with $A_X\asymp X^{1/3}$ and $B_X\asymp X^{1/2}$, so dividing through by $\NN(X)$ gives
\[
\frac{N_p(X)}{\NN(X)}
\ll
\frac{1}{p^2}+\frac{1}{pX^{1/3}}+\frac{1}{X^{1/2}}.
\]

For the second statement, note that $\Ht(E)\le X$ implies
\[
|4a_4^3+27a_6^2|\le X+X=2X,
\]
and hence we obtain
\[
|\Delta(E)|\le 32X.
\]
Therefore, if $\tau_p(E)\notin\KK_0$, then $p^2\mid \Delta(E)\neq 0$, so necessarily $p\le (32X)^{1/2}$. Summing the first bound over primes $Y<p\le (32X)^{1/2}$, we obtain
\[
\frac{1}{\NN(X)}
\sum_{\Ht(E)\le X}\sum_{p>Y}\1_{\tau_p(E)\notin\KK_0}
\ll
\sum_{Y<p\le (32X)^{1/2}}\frac{1}{p^2}
+
\frac{1}{X^{1/3}}\sum_{Y<p\le (32X)^{1/2}}\frac{1}{p}
+
\frac{\pi((32X)^{1/2})}{X^{1/2}}.
\]
The second term is $O(X^{-1/3}\log X)$, and the third is $O(1/\log X)$ by the standard estimate $\pi(T)\ll T/\log T$. Taking the limit superior as $X\to\infty$ gives
\[
\limsup_{X\to\infty}
\frac{1}{\NN(X)}
\sum_{\Ht(E)\le X}\sum_{p>Y}\1_{\tau_p(E)\notin\KK_0}
\ll
\sum_{p>Y}\frac{1}{p^2},
\]
and the right-hand side tends to $0$ as $Y\to\infty$.
\end{proof}

\subsection{Finite-prime factorization}

We now isolate the finite-prime multiplicativity that underlies the Euler product in \cite{GOT}, recast at the level of local reduction data.

\begin{lemma}\label{lem:local-constancy}
Fix a prime $p$, and let
\[
U_p:=\{(a_4,a_6)\in\Z_p^2:\Delta(a_4,a_6)\neq 0\}.
\]
Then the map
\[
(a_4,a_6)\longmapsto \tau_p(E(a_4,a_6))
\]
is locally constant on $U_p$. Here $\tau_p$ denotes the local reduction datum, including the finite splitting and component-group data determining $c_p$. Consequently, for every bounded function $f:\KK\to\C$, the composition $f(\tau_p(\cdot))$ is bounded and continuous on $U_p$. The complement $\Z_p^2\setminus U_p$ has Haar measure zero.
\end{lemma}

\begin{proof}
Fix $x=(a_4,a_6)\in U_p$. Tate's algorithm, applied to the $p$-integral Weierstrass equation determined by $x$, reaches the Kodaira symbol, the relevant splitting alternatives, and the component-group size after finitely many steps. Each step examines only finitely many valuations and residue conditions on the current coefficients and on the auxiliary polynomials appearing in the multiplicative and additive cases; see \cite[Chapter IV, \S9]{SilvermanAdvanced}, \cite{Tate}, and the case-by-case analysis in Section~2 of \cite{GOT}. Once the relevant case and the corresponding splitting tests are fixed, all of those finitely many tests remain unchanged on a sufficiently small $p$-adic neighborhood of $x$. Hence the local reduction datum $\tau_p$ is constant on that neighborhood. This proves local constancy on $U_p$.

Finally, the singular locus $\Delta=0$ has Haar measure zero in $\Z_p^2$: for each fixed $a_6\in\Z_p$, the equation $4a_4^3+27a_6^2=0$ has at most three solutions in $a_4$, so Fubini's theorem gives measure zero; see \cite[Chapter~8]{RudinRCA}.
\end{proof}

\begin{proposition}\label{prop:finite-prime}
Let $S$ be a finite set of primes. For each $p\in S$, let $f_p:\KK\to\C$ be bounded. Then the natural density
\[
\lim_{X\to\infty}\frac{1}{\NN(X)}\sum_{\Ht(E)\le X}\prod_{p\in S} f_p(\tau_p(E))
\]
exists and equals
\begin{equation}\label{eq:finite-prime-factorization}
\prod_{p\in S}\left(\sum_{K\in\KK}\delta_p(K)f_p(K)\right).
\end{equation}
\end{proposition}

\begin{proof}
Let $\mu_p$ be normalized Haar measure on $\Z_p^2$, and let
\[
K_S:=\prod_{p\in S}\Z_p^2
\qquad {\text {\rm and}}\qquad
\mu_S:=\bigotimes_{p\in S}\mu_p.
\]
For each $X$, define a probability measure $\mu_{S,X}$ on $K_S$ by
\[
\mu_{S,X}:=
\frac{1}{\NN(X)}
\sum_{\Ht(E(a_4,a_6))\le X}
\delta_{((a_4,a_6))_{p\in S}},
\]
where the sum runs over nonsingular pairs and $\delta_y$ denotes the unit point mass at $y$.

If $C\subseteq K_S$ is a cylinder set cut out by congruence conditions modulo $p^{\nu_p}$ for $p\in S$, then $C$ corresponds to a subset $A\subseteq (\Z/M\Z)^2$ with
\[
M:=\prod_{p\in S}p^{\nu_p}.
\]
Lemma~\ref{lem:equidistribution} therefore gives
\[
\mu_{S,X}(C)\to \mu_S(C)
\qquad (X\to\infty).
\]
Cylinder sets are clopen, so their indicators are continuous. Finite linear combinations of these indicators form the algebra of locally constant functions on $K_S$, which separates points and contains the constants. By the Stone--Weierstrass theorem, this algebra is uniformly dense in $C(K_S)$; see \cite[Theorem~7.33]{RudinPMA}. Hence $\mu_{S,X}$ converges weakly to $\mu_S$.

Now define
\[
G((x_p)_{p\in S}) := \prod_{p\in S} f_p(\tau_p(x_p))
\]
on $\prod_{p\in S}U_p$, where $U_p$ is as in Lemma~\ref{lem:local-constancy}, and extend $G$ arbitrarily to all of $K_S$. By Lemma~\ref{lem:local-constancy}, $G$ is bounded and continuous on $\prod_{p\in S}U_p$. Its discontinuity set is contained in
\[
K_S\setminus \prod_{p\in S}U_p,
\]
which has $\mu_S$-measure $0$. Since $\Z_p^2\setminus U_p$ has Haar measure zero for each $p\in S$, the values of the extension on the singular locus do not affect the integral. The Portmanteau theorem for bounded functions whose discontinuity set has $\mu_S$-measure zero therefore gives
\[
\int_{K_S} G\,d\mu_{S,X}\to \int_{K_S} G\,d\mu_S;
\]
see \cite[Section~2]{Billingsley}.
The left side is exactly
\[
\frac{1}{\NN(X)}\sum_{\Ht(E)\le X}\prod_{p\in S} f_p(\tau_p(E)).
\]
Since $\mu_S$ is a product measure and $G$ factors prime-by-prime, the right side equals
\[
\prod_{p\in S}
\left(
\int_{\Z_p^2} f_p(\tau_p(x_p))\,d\mu_p(x_p)
\right).
\]
By the definition of the local densities, each local integral is
\[
\int_{\Z_p^2} f_p(\tau_p(x_p))\,d\mu_p(x_p)
=
\sum_{K\in\KK}\delta_p(K)f_p(K).
\]
This proves \eqref{eq:finite-prime-factorization}.
\end{proof}

\section{Proof of Theorem~\ref{thm:master}}\label{sec:proofs}
This section proves Theorem~\ref{thm:master}. We first truncate every statistic to a finite set $S$ of primes, where the limiting generating function factors over $S$ by Proposition~\ref{prop:finite-prime}. We then let $S$ grow, controlling the primes outside $S$ with Lemma~\ref{lem:uniform-large-prime-tail}.

\subsection{Truncated statistics}
We replace each statistic by a truncated version that only sees a finite set of primes. For a fixed curve the two agree once that set is large enough, and the limiting generating function of the truncated statistics is a finite Euler product (Lemma~\ref{lem:truncated-gf}).

For a finite set of primes $S$, define the truncated statistics
\[
\omega_{K,S}(E):=\sum_{p\in S}\1_{\tau_p(E)=K}\qquad (K\in\KK),
\]
\[
\omega_{\Tam,S}(E):=\sum_{p\in S}\1_{c_p(E)>1},
\qquad
\Omega_S(E):=\sum_{p\in S}\Omega(c_p(E)),
\]
\[
\mathbf v_{\Pi,S}(E):=\Bigl(\sum_{p\in S}v_\ell(c_p(E))\Bigr)_{\ell\in\Pi},
\qquad
\Tam_S(E):=\prod_{p\in S}c_p(E).
\]
Because only finitely many primes divide the discriminant of a fixed elliptic curve, once $S$ contains every prime at which $\tau_p(E)\notin\KK_0$, we have
\[
\begin{aligned}
\omega_{K,S}(E)&= \omega_K(E)\quad(K\notin\KK_0),
&\omega_{\Tam,S}(E)&= \omega_{\Tam}(E),\\
\Omega_S(E)&= \Omega(\Tam(E)),
&\mathbf v_{\Pi,S}(E)&= \mathbf v_\Pi(\Tam(E)),\\
&&\Tam_S(E)&= \Tam(E).
\end{aligned}
\]
In particular, each truncated statistic converges to its global counterpart as $S$ increases through finite sets of primes.

\begin{lemma}\label{lem:truncated-gf}
Let $\Lambda\subseteq\KK\setminus\KK_0$ be a finite set of local reduction data, let $\Pi$ be a finite set of primes, and let $S$ be a finite set of primes. If $\Reals(s)\ge 0$ and
\[
|u|\le 1,\qquad |w|\le 1,\qquad |z_\ell|\le 1\ \forall\ell\in\Pi,\qquad |u_K|\le 1\ \forall K\in\Lambda,
\]
then
\begin{align}\label{eq:truncated-identity}
&\lim_{X\to\infty}\frac{1}{\NN(X)}\sum_{\Ht(E)\le X}
u^{\omega_{\Tam,S}(E)}w^{\Omega_S(E)}\mathbf z^{\mathbf v_{\Pi,S}(E)}\prod_{K\in\Lambda}u_K^{\omega_{K,S}(E)}\Tam_S(E)^{-s}
\notag\\
&\hspace{1.6cm}=
\prod_{p\in S}
\!\left(\beta_p+\sum_{K\in\KK\setminus\KK_0}\delta_p(K)\,\Phi_p(K;s,u,w,\mathbf z,\mathbf u)\right).
\end{align}
\end{lemma}

\begin{proof}
Apply Proposition~\ref{prop:finite-prime} with
\[
f_p(K)=\Phi_p(K;s,u,w,\mathbf z,\mathbf u)
\qquad (p\in S).
\]
Since $\Reals(s)\ge 0$ and all variables have modulus at most $1$, each $f_p$ is bounded by $1$ in absolute value, so Proposition~\ref{prop:finite-prime} applies. On $K\in\KK_0$ we have $\Phi_p(K)=1$, and hence
\[
\sum_{K\in\KK}\delta_p(K)\Phi_p(K)=\beta_p+\sum_{K\notin\KK_0}\delta_p(K)\Phi_p(K),
\]
which gives \eqref{eq:truncated-identity}.
\end{proof}

\subsection{Proof of the master Euler product}
We now let the finite set of primes grow. The truncation error splits into the finitely many primes below a cutoff, which Proposition~\ref{prop:finite-prime} handles, and the primes above it, which Lemma~\ref{lem:uniform-large-prime-tail} handles; letting first the cutoff and then the set of primes grow gives the theorem.

\begin{proof}[Proof of Theorem~\ref{thm:master}]
Fix a finite set of primes $S$. By Lemma~\ref{lem:truncated-gf}, the density generating function for the truncated statistics is exactly the finite product on the right side of \eqref{eq:truncated-identity}.

By Lemma~\ref{lem:local-estimates}, the infinite product in \eqref{eq:master-product} converges absolutely because its local factors are $1+O(p^{-2})$, and $\sum_p p^{-2}<\infty$. In particular, it converges uniformly on every compact subset of the closed parameter set
\[
\{\Reals(s)\ge 0,\ |u|\le 1,\ |w|\le 1,\ |z_\ell|\le 1\ \forall\ell\in\Pi,\ |u_K|\le 1\ \forall K\in\Lambda\}.
\]

Let
\[
W_S(E):=u^{\omega_{\Tam,S}(E)}w^{\Omega_S(E)}\mathbf z^{\mathbf v_{\Pi,S}(E)}\prod_{K\in\Lambda}u_K^{\omega_{K,S}(E)}\Tam_S(E)^{-s}
\]
and
\[
W(E):=u^{\omega_{\Tam}(E)}w^{\Omega(\Tam(E))}\mathbf z^{\mathbf v_\Pi(\Tam(E))}\prod_{K\in\Lambda}u_K^{\omega_K(E)}\Tam(E)^{-s}.
\]
Because $\Reals(s)\ge 0$ and all variables have modulus at most $1$, we have
\[
|W_S(E)|\le 1,
\qquad
|W(E)|\le 1.
\]
Moreover $W_S(E)=W(E)$ unless there exists a prime $p\notin S$ with $\tau_p(E)\notin\KK_0$. Hence
\[
|W_S(E)-W(E)|\le 2\sum_{p\notin S}\1_{\tau_p(E)\notin\KK_0}.
\]

Write
\[
L_p(s;u,w,\mathbf z,\mathbf u):=
\beta_p+\sum_{K\in\KK\setminus\KK_0}\delta_p(K)\,\Phi_p(K;s,u,w,\mathbf z,\mathbf u).
\]
Then, for every fixed choice of parameters in the theorem,
\begin{align*}
\limsup_{X\to\infty}\bigl|\mathcal Z_{\Lambda,\Pi,X}&(s;u,w,\mathbf z,\mathbf u)-\prod_p L_p(s;u,w,\mathbf z,\mathbf u)\bigr| \\
&\le
\limsup_{X\to\infty}\left|\frac{1}{\NN(X)}\sum_{\Ht(E)\le X}W(E)-\frac{1}{\NN(X)}\sum_{\Ht(E)\le X}W_S(E)\right|\\
&\quad\quad+
\limsup_{X\to\infty}\left|\frac{1}{\NN(X)}\sum_{\Ht(E)\le X}W_S(E)-\prod_{p\in S}L_p(s;u,w,\mathbf z,\mathbf u)\right|\\
&\quad \quad+
\left|\prod_{p\in S}L_p(s;u,w,\mathbf z,\mathbf u)-\prod_p L_p(s;u,w,\mathbf z,\mathbf u)\right|.
\end{align*}
The middle term is $0$ by Lemma~\ref{lem:truncated-gf}. For the first term, enlarge $S$ if necessary so that $2,3\in S$, and fix a real number $Y>\max S$ with $Y\ge5$. Then
\begin{align*}
\frac{1}{\NN(X)}\sum_{\Ht(E)\le X}|W(E)-W_S(E)|
&\le
2\sum_{\substack{p\le Y\\ p\notin S}}
\frac{\#\{E:\Ht(E)\le X,\ \tau_p(E)\notin\KK_0\}}{\NN(X)}\\
&\quad+
2\frac{1}{\NN(X)}
\sum_{\Ht(E)\le X}\sum_{p>Y}\1_{\tau_p(E)\notin\KK_0}.
\end{align*}
For each fixed prime $p$, Proposition~\ref{prop:finite-prime} with $S=\{p\}$ and $f_p(K)=\1_{K\notin\KK_0}$ gives
\[
\lim_{X\to\infty}
\frac{\#\{E:\Ht(E)\le X,\ \tau_p(E)\notin\KK_0\}}{\NN(X)}
=
1-\beta_p.
\]
Summing over the finitely many primes $p\le Y$ with $p\notin S$, we obtain
\[
\lim_{X\to\infty}
\sum_{\substack{p\le Y\\ p\notin S}}
\frac{\#\{E:\Ht(E)\le X,\ \tau_p(E)\notin\KK_0\}}{\NN(X)}
=
\sum_{\substack{p\le Y\\ p\notin S}}(1-\beta_p).
\]
For the tail over $p>Y$, Lemma~\ref{lem:uniform-large-prime-tail} yields
\[
\limsup_{X\to\infty}
\frac{1}{\NN(X)}
\sum_{\Ht(E)\le X}\sum_{p>Y}\1_{\tau_p(E)\notin\KK_0}
\ll
\sum_{p>Y}\frac{1}{p^2}.
\]
Hence, we have
\[
\limsup_{X\to\infty}
\frac{1}{\NN(X)}\sum_{\Ht(E)\le X}|W(E)-W_S(E)|
\le
2\sum_{\substack{p\le Y\\ p\notin S}}(1-\beta_p)
+
O\!\left(\sum_{p>Y}\frac{1}{p^2}\right).
\]
Letting $Y\to\infty$ and using Lemma~\ref{lem:K0-asymptotic}, which gives $1-\beta_p\ll p^{-2}$, we obtain
\[
\limsup_{X\to\infty}
\frac{1}{\NN(X)}\sum_{\Ht(E)\le X}|W(E)-W_S(E)|
\le
2\sum_{p\notin S}(1-\beta_p).
\]
For the third term, absolute convergence of the Euler product and Lemma~\ref{lem:local-estimates} give
\[
\left|\prod_{p\in S}L_p(s;u,w,\mathbf z,\mathbf u)-\prod_p L_p(s;u,w,\mathbf z,\mathbf u)\right|
\to 0
\qquad (S\to\infty).
\]
Since $\sum_p(1-\beta_p)<\infty$ by Lemma~\ref{lem:K0-asymptotic}, letting $S$ increase through finite sets of primes proves \eqref{eq:master-product}. The five specializations are immediate.
\end{proof}

\section{Analytic lemmas for coefficient extraction}\label{sec:analytic}

Theorem~\ref{thm:master} is proved entirely within Section~\ref{sec:proofs}. The lemmas in this section are used only to pass from the pointwise generating-function limits of Theorem~\ref{thm:master} to coefficientwise limiting distributions in the later theorems.

\subsection{Normal-family upgrade}

We shall repeatedly upgrade the pointwise limits furnished by Theorem~\ref{thm:master} to locally uniform limits inside the open unit discs and polydiscs, via the classical Vitali--Montel theorem.

\begin{lemma}\label{lem:vitali}
Let $D$ be either the open unit disc in $\C$ or a finite polydisc
\[
D=\{(z_1,\dots,z_n)\in\C^n:\ |z_i|<1\}.
\]
Let $F_X$ be holomorphic on $D$, and suppose that $|F_X|\le 1$ on $D$ for every $X$. If $F_X$ converges pointwise on $D$ to a function $F$, then $F$ is holomorphic and the convergence is uniform on every compact subset of $D$.
\end{lemma}

\begin{proof}
This is the classical Vitali--Montel theorem in one variable, together with its standard polydisc version in several variables; see, for example, \cite[Chapter~VII]{Conway} for the one-variable normal-family theorem and \cite[Chapter~1]{KrantzSCV} for the several-variable extension.
\end{proof}

\subsection{Coefficient extraction}

The next lemma turns convergence of generating functions into convergence of coefficient densities.

\begin{lemma}\label{lem:cauchy-coeff}
Let $F_X(u)=\sum_{n\ge 0}a_{X,n}u^n$ be power series that converge on $|u|\le 1$, and suppose that for every $0<\rho<1$ the functions $F_X$ converge uniformly on $|u|=\rho$ to a limit $F(u)=\sum_{n\ge 0}a_nu^n$. Then, for each fixed $n\ge 0$,
\[
a_{X,n}\to a_n
\qquad (X\to\infty).
\]
If, in addition, $I$ is a finite index set and
\[
G_X(\mathbf z)=\sum_{\mathbf j\in\Z_{\ge 0}^{I}} b_{X,\mathbf j}\mathbf z^{\mathbf j}
\]
converges uniformly on every polytorus $|z_i|=\rho_i<1$ $(i\in I)$ to
$G(\mathbf z)=\sum_{\mathbf j\in\Z_{\ge0}^I}b_{\mathbf j}\mathbf z^{\mathbf j}$, then for each fixed multi-index $\mathbf j$ one has $b_{X,\mathbf j}\to b_{\mathbf j}$.
\end{lemma}

\begin{proof}
For the one-variable statement, Cauchy's coefficient formula gives
\[
a_{X,n}=
\frac{1}{2\pi i}\int_{|u|=\rho}\frac{F_X(u)}{u^{n+1}}\,du
\qquad \text{and}\qquad
a_n=
\frac{1}{2\pi i}\int_{|u|=\rho}\frac{F(u)}{u^{n+1}}\,du.
\]
Uniform convergence on the circle $|u|=\rho$ therefore implies $a_{X,n}\to a_n$. The multivariable version follows from the iterated Cauchy integral formula on the polytorus $|z_i|=\rho_i$; see \cite[Chapter~4]{Ahlfors} for the one-variable Cauchy coefficient formula and \cite[Chapter~1]{KrantzSCV} for the several-variable form.
\end{proof}

\section{Proofs of the distribution, valuation, covariance, and tail theorems}\label{sec:distribution-proofs}
In this section we compile the results of the previous sections analytically to obtain the main results of this paper.

\subsection{Proof of the distribution theorem for the counts \texorpdfstring{$\omega_K(E)$}{omega K(E)}}
Specializing Theorem~\ref{thm:master} to the variables $u_K$ gives the joint generating function of the counts $\omega_K(E)$, $K\in\Lambda$. Since $\sum_p\delta_p(K)<\infty$ for every $K\in\Lambda$, the resulting Euler product is entire, and its logarithmic derivatives give the means, variances and covariances.

\begin{proof}[Proof of Theorem~\ref{thm:omegaK}]
Apply Theorem~\ref{thm:master} with the given $\Lambda$, $\Pi=\varnothing$, $u=w=1$, and $s=0$. This gives pointwise convergence
\[
F_X(\mathbf u):=\frac{1}{\NN(X)}\sum_{\Ht(E)\le X}\prod_{K\in\Lambda}u_K^{\omega_K(E)}
\longrightarrow
F(\mathbf u):=
\prod_p\!\left(1-\sum_{K\in\Lambda}\delta_p(K)+\sum_{K\in\Lambda}\delta_p(K)u_K\right)
\]
for every $\mathbf u$ in the open unit polydisc. Since $|F_X(\mathbf u)|\le 1$ there, Lemma~\ref{lem:vitali} upgrades this to uniform convergence on every closed polydisc $|u_K|\le\rho<1$. Because
\[
1-\sum_{K\in\Lambda}\delta_p(K)+\sum_{K\in\Lambda}\delta_p(K)u_K=1+\sum_{K\in\Lambda}\delta_p(K)(u_K-1),
\]
and $\sum_p\delta_p(K)<\infty$ for each $K\in\Lambda$ by Lemma~\ref{lem:local-estimates}, the product converges absolutely and locally uniformly for all $\mathbf u\in\C^\Lambda$, so the identity extends to all $\mathbf u\in\C^\Lambda$. The coefficient of $\mathbf u^{\mathbf r}$ in $F_X(\mathbf u)$ is exactly $\#\{E:\Ht(E)\le X,\ \omega_K(E)=r_K\ \forall K\in\Lambda\}/\NN(X)$, so Lemma~\ref{lem:cauchy-coeff} shows that $\pi_\Lambda(\mathbf r)$ exists and equals the corresponding coefficient of $F(\mathbf u)$. This proves \eqref{eq:omegaK-gf}.

For the marginal statistics, fix $K\in\Lambda$ and apply the one-variable version of the above with $\Lambda=\{K\}$. Set $a_p:=\delta_p(K)$, so that $F_K(u)=\prod_p(1+a_p(u-1))$ and $\sum_p a_p<\infty$. Since $0\le a_p\le 1$, on the disc $|u-1|\le 1/2$ we have $|a_p(u-1)|\le 1/2$ for every $p$, so each factor stays away from $0$. Termwise differentiation of $\log F_K=\sum_p\log(1+a_p(u-1))$ gives
\[
F_K'(1)=\sum_p a_p=\sum_p\delta_p(K),
\qquad
F_K''(1)=\Bigl(\sum_p a_p\Bigr)^{\!2}-\sum_p a_p^2.
\]
Since $F_K(u)=\E[u^{\omega_K(E)}]$ is a probability generating function and $F_K'(1)<\infty$, we have $\E[\omega_K(E)]=F_K'(1)$, and
\[
\Var(\omega_K(E))=F_K''(1)+F_K'(1)-F_K'(1)^2=\sum_p a_p(1-a_p).
\]
This proves \eqref{eq:omegaK-mean-var}.

Finally, for distinct $K,K'\in\Lambda$, logarithmic differentiation of the bivariate generating function $F_{K,K'}(u_K,u_{K'})=\prod_p(1+a_p(u_K-1)+a_p'(u_{K'}-1))$, where $a_p:=\delta_p(K)$ and $a_p':=\delta_p(K')$, gives
\[
\frac{\partial^2 F_{K,K'}}{\partial u_K\,\partial u_{K'}}\bigg|_{(1,1)}
=\sum_p\frac{\partial}{\partial u_{K'}}\left(\frac{a_p}{1+a_p(u_K-1)+a_p'(u_{K'}-1)}\right)\bigg|_{(1,1)}
+\Bigl(\sum_p a_p\Bigr)\Bigl(\sum_p a_p'\Bigr).
\]
Evaluating at $(u_K,u_{K'})=(1,1)$ and applying the standard joint-p.g.f.~identity
\[
\Cov(\omega_K,\omega_{K'})=\frac{\partial^2 F_{K,K'}}{\partial u_K\,\partial u_{K'}}\bigg|_{(1,1)}-\Bigl(\sum_p a_p\Bigr)\Bigl(\sum_p a_p'\Bigr),
\]
we obtain, after expanding the partial derivative and using $(u_K,u_{K'})=(1,1)$,
\[
\Cov(\omega_K(E),\omega_{K'}(E))=-\sum_p a_p a_p'=-\sum_p\delta_p(K)\delta_p(K'),
\]
which is \eqref{eq:omegaK-cov}. Absolute convergence follows from $a_p,a_p'\ll p^{-2}$.
\end{proof}

\subsection{Proof of the distribution theorem for \texorpdfstring{$\omega_{\Tam}(E)$}{the Tamagawa prime count}}
The argument is the one-variable analogue of the previous one, with the single variable $u$ recording whether $c_p(E)>1$; now infinitely many local reduction data are aggregated at once.

\begin{proof}[Proof of Theorem~\ref{thm:omegaE}]
Take $\Lambda=\varnothing$, $\Pi=\varnothing$, set $w=1$, and put $s=0$ in Theorem~\ref{thm:master}. This gives pointwise convergence
\[
F_X(u):=\frac{1}{\NN(X)}\sum_{\Ht(E)\le X}u^{\omega_{\Tam}(E)}
\longrightarrow
F(u):=
\prod_p\bigl(\delta_p(1)+(1-\delta_p(1))u\bigr)
\]
for every $|u|<1$. Since $|F_X(u)|\le 1$ on the open unit disc, Lemma~\ref{lem:vitali} upgrades this to uniform convergence on every closed disc $|u|\le \rho<1$. Since the coefficient of $u^r$ in $F_X(u)$ is exactly
\[
\frac{\#\{E:\Ht(E)\le X,\ \omega_{\Tam}(E)=r\}}{\NN(X)},
\]
Lemma~\ref{lem:cauchy-coeff} shows that the density $\pi_r$ exists for every $r\ge 0$, and that
\[
\sum_{r\ge 0}\pi_ru^r=F(u)
\qquad (|u|<1).
\]
Because
\[
\delta_p(1)+(1-\delta_p(1))u = 1 + (u-1)(1-\delta_p(1))
\]
and $\sum_p(1-\delta_p(1))<\infty$ by Lemma~\ref{lem:local-estimates}, the product converges absolutely and locally uniformly for all $u\in\C$, so the identity extends to all $u\in\C$. This proves \eqref{eq:omegaE-gf}.

Write
\[
a_p:=1-\delta_p(1).
\]
Then $\sum_p a_p<\infty$, and $F(u)=\prod_p(1+a_p(u-1))$. On the disc $|u-1|\le 1/2$ we have $|a_p(u-1)|\le 1/2$ for every $p$, so each factor stays away from $0$. Termwise differentiation gives
\[
F'(1)=\sum_p a_p=\sum_p(1-\delta_p(1)),
\qquad
F''(1)=\Bigl(\sum_pa_p\Bigr)^{\!2}-\sum_pa_p^2.
\]
Since $F(u)=\sum_r\pi_ru^r$ is a probability generating function and $F'(1)<\infty$, we have
\[
\E[\omega_{\Tam}(E)]=F'(1),\qquad
\Var(\omega_{\Tam}(E))=F''(1)+F'(1)-F'(1)^2=\sum_p a_p(1-a_p).
\]
Recalling that $1-a_p=\delta_p(1)$ gives \eqref{eq:omegaE-mean-var}. Absolute convergence follows from $a_p\ll p^{-2}$.
\end{proof}

\subsection{Proof of the distribution theorem for \texorpdfstring{$\Omega(\Tam(E))$}{Omega(Tam(E))}}
Here the variable $w$ records $\Omega(\Tam(E))$. We identify the limiting law and its mean, and compare that mean with the mean of $\omega(\Tam(E))$ under the limiting Tamagawa law.

\begin{proof}[Proof of Theorem~\ref{thm:Omega}]
Take $\Lambda=\varnothing$, $\Pi=\varnothing$, set $u=1$, and put $s=0$ in Theorem~\ref{thm:master}. Then we have pointwise convergence
\[
G_X(w):=\frac{1}{\NN(X)}\sum_{\Ht(E)\le X}w^{\Omega(\Tam(E))}
\longrightarrow
G(w):=
\prod_p\left(\sum_{t\ge 1}\delta_p(t)w^{\Omega(t)}\right)
\]
for every $|w|<1$. Since $|G_X(w)|\le 1$ on the open unit disc, Lemma~\ref{lem:vitali} upgrades this to uniform convergence on every closed disc $|w|\le \rho<1$. The coefficient of $w^b$ in $G_X(w)$ is exactly
\[
\frac{\#\{E:\Ht(E)\le X,\ \Omega(\Tam(E))=b\}}{\NN(X)},
\]
so Lemma~\ref{lem:cauchy-coeff} shows that the density $\rho_b$ exists for every $b\ge 0$, and that
\[
\sum_{b\ge 0}\rho_bw^b=G(w)
\qquad (|w|<1).
\]
For $|w|\le 1$, each local factor is continuous and satisfies
\[
\left|\sum_{t\ge 1}\delta_p(t)w^{\Omega(t)}-1\right|
\le 2(1-\delta_p(1)).
\]
Lemma~\ref{lem:local-estimates} therefore gives uniform convergence of the Euler product on the closed unit disc. In particular, $G(w)\to G(1)=1$ as real $w\uparrow 1$. Since $\rho_b\ge 0$, the monotone convergence theorem now gives
\[
\sum_{b\ge 0}\rho_b
=\lim_{w\uparrow 1}\sum_{b\ge 0}\rho_bw^b
=\lim_{w\uparrow 1}G(w)=1.
\]
Thus the power series on the left converges absolutely and uniformly for $|w|\le 1$, and the identity extends to the closed unit disc by continuity. This gives \eqref{eq:Omega-gf}.

For each prime $p$, put
\[
h_p(w):=\sum_{t\ge 1}\delta_p(t)w^{\Omega(t)}
\qquad (0\le w\le 1).
\]
Then $h_p(1)=1$, and for $0<w\le 1$ we have
\[
h_p'(w)=\sum_{t\ge 1}\delta_p(t)\Omega(t)w^{\Omega(t)-1},
\]
which extends continuously to $w=0$. Hence, with
\[
A_p:=\sup_{0\le w\le 1}|h_p'(w)|,
\]
we have $A_p\ll p^{-2}$ for $p\ge 5$. In particular, $\sum_p A_p<\infty$ and $A:=\sup_p A_p<\infty$.

Choose $0<\eta\le 1/(2(1+A))$. For $w\in[1-\eta,1]$, the mean-value theorem gives
\[
|h_p(w)-1|\le A_p|1-w|\le A\eta\le \tfrac12,
\]
so $h_p(w)\in[\tfrac12,\tfrac32]$. Therefore $\log h_p(w)$ is well defined on $[1-\eta,1]$, the series
\[
\sum_p \log h_p(w)
\]
converges uniformly there, and
\[
\left|\frac{h_p'(w)}{h_p(w)}\right|\le 2A_p
\]
shows that
\[
\sum_p \frac{h_p'(w)}{h_p(w)}
\]
also converges uniformly on $[1-\eta,1]$. Hence $\log G(w)=\sum_p\log h_p(w)$ is differentiable at $w=1$, and since $G(1)=1$ and $h_p(1)=1$ for every $p$,
\[
G'(1)=G(1)\sum_p\frac{h_p'(1)}{h_p(1)}
=\sum_p h_p'(1)
=
\sum_p\sum_{t\ge 1}\delta_p(t)\Omega(t).
\]
Since $G(w)=\sum_{b\ge 0}\rho_bw^b$ is a probability generating function and $G'(1)<\infty$, we obtain
\[
\E[\Omega(\Tam(E))]=\sum_{b\ge 0}b\rho_b=G'(1),
\]
which proves \eqref{eq:Omega-mean}.

Finally, $\omega(n)\le \Omega(n)$ for every positive integer $n$, so
\[
\omega(\Tam(E))\le \Omega(\Tam(E))
\]
for every elliptic curve $E$. Taking expectations yields \eqref{eq:smallomega-bound}.
\end{proof}

\subsection{Proof of the valuation theorem}
The variables $z_\ell$, for $\ell$ in a finite set of primes $\Pi$, record the $\ell$-adic valuations of $\Tam(E)$. Setting some of them to $0$ gives the densities of curves whose Tamagawa product is coprime to prescribed primes.

\begin{proof}[Proof of Theorem~\ref{thm:valuations}]
Set $\Lambda=\varnothing$, $u=1$, $w=1$, and $s=0$ in Theorem~\ref{thm:master}. Then we have pointwise convergence
\[
H_X(\mathbf z):=\frac{1}{\NN(X)}\sum_{\Ht(E)\le X}\mathbf z^{\mathbf v_\Pi(\Tam(E))}
\longrightarrow
H(\mathbf z):=
\prod_p\left(\sum_{t\ge 1}\delta_p(t)\prod_{\ell\in\Pi}z_\ell^{v_\ell(t)}\right)
\]
for every $|z_\ell|<1$. Since $|H_X(\mathbf z)|\le 1$ on the open unit polydisc, Lemma~\ref{lem:vitali} upgrades this to uniform convergence on every closed polydisc $|z_\ell|\le \rho<1$. The coefficient of $\mathbf z^{\mathbf j}$ in $H_X(\mathbf z)$ is exactly
\[
\frac{\#\{E:\Ht(E)\le X,\ v_\ell(\Tam(E))=j_\ell\ \text{for all }\ell\in\Pi\}}{\NN(X)}.
\]
Therefore Lemma~\ref{lem:cauchy-coeff} implies the existence of the density $Q_\Pi(\mathbf j)$ for every multi-index $\mathbf j$. By Fatou's lemma $\sum_{\mathbf j}Q_\Pi(\mathbf j)\le 1$, so the power series on the left converges absolutely and uniformly on the closed unit polydisc $|z_\ell|\le 1$. For the Euler product, the factors at $p=2,3$ are absolutely convergent power series on the closed polydisc, since $\sum_t\delta_p(t)=1$. For $p\ge 5$, Lemma~\ref{lem:local-estimates} gives
\[
\sup_{|z_\ell|\le 1}
\left|
\sum_{t\ge 1}\delta_p(t)\prod_{\ell\in\Pi}z_\ell^{v_\ell(t)}-1
\right|
\le
2\sum_{t\ge 2}\delta_p(t)
\ll p^{-2}.
\]
Thus the Weierstrass $M$-test gives uniform convergence of the Euler product on the closed unit polydisc. Hence the identity extends by continuity from the open polydisc to $|z_\ell|\le 1$, which proves \eqref{eq:valuation-gf}. Setting $z_\ell=1$ for every $\ell\in\Pi$ then gives $\sum_{\mathbf j}Q_\Pi(\mathbf j)=1$.

To prove \eqref{eq:noA-divisibility}, apply the theorem with a finite set $\Pi$ containing $A$, and then set $z_\ell=0$ for $\ell\in A$ and $z_\ell=1$ for $\ell\in\Pi\setminus A$. On the left, only those curves survive for which $v_\ell(\Tam(E))=0$ for every $\ell\in A$, that is, for which no prime in $A$ divides $\Tam(E)$. On the right, the local factor becomes
\[
\sum_{\substack{t\ge 1 \\ (t,\prod_{\ell\in A}\ell)=1}}\delta_p(t),
\]
which proves \eqref{eq:noA-divisibility}. Equation \eqref{eq:ell-divides-Tam} is the complement of the case $A=\{\ell\}$, and \eqref{eq:odd-density} is the special case $A=\{2\}$.
\end{proof}

\subsection{Proof of the anti-correlation theorem}
Part~(1) computes the covariance of two valuations from their joint limiting law. Part~(2) is proved prime by prime: every local term of the covariance is negative, because the cross moment is supported on the geometric tail $t\ge5$ and is dominated there by a product of two explicit lower bounds. Part~(3) follows from the same local bounds.

\begin{proof}[Proof of Theorem~\ref{thm:covariance-vl}]
Fix distinct primes $\ell,\ell'$, and set $\Pi=\{\ell,\ell'\}$, $\Lambda=\varnothing$, $u=w=1$, $s=0$ in Theorem~\ref{thm:master}. The joint valuation generating function is
\[
H(z_\ell,z_{\ell'})=\prod_p h_p(z_\ell,z_{\ell'}),
\qquad
h_p(z_\ell,z_{\ell'}):=\sum_{t\ge 1}\delta_p(t)z_\ell^{v_\ell(t)}z_{\ell'}^{v_{\ell'}(t)}.
\]
For $p\ge 5$, the estimates
\[
\sum_{t\ge 1}\delta_p(t)v_\ell(t)\ll p^{-2},
\qquad
\sum_{t\ge 1}\delta_p(t)v_\ell(t)v_{\ell'}(t)\le
\sum_{t\ge 2}\delta_p(t)t^2\ll p^{-2}
\]
follow from Lemma~\ref{lem:local-estimates} and \eqref{eq:moment-local-estimate} with $k=2$.

For a complex neighborhood of $(1,1)$, note that $v_\ell(t)+v_{\ell'}(t)\le\Omega(t)$ because $\ell\ne\ell'$. Thus, on $|z_\ell|,|z_{\ell'}|\le2$, the series for $h_p-1$ and its derivatives of total order at most two are bounded by a constant times $\sum_{t\ge2}\delta_p(t)t^3$. This is finite for every $p$ by the geometric tails in Theorem~\ref{thm:GOT-local}, and it is $O(p^{-2})$ for $p\ge5$ by Lemma~\ref{lem:local-estimates}. For large $p$ these bounds keep $h_p$ uniformly close to $1$; the remaining finitely many factors are nonzero on a sufficiently small common neighborhood of $(1,1)$, since $h_p(1,1)=1$. Hence the logarithmic derivatives of $\prod_p h_p$ up to the mixed second derivative converge uniformly on a smaller closed neighborhood of $(1,1)$. Termwise logarithmic differentiation gives
\[
\frac{\partial^2\log H}{\partial z_\ell\partial z_{\ell'}}\bigg|_{(1,1)}
=
\sum_p\Bigl(
\sum_{t\ge 1}\delta_p(t)v_\ell(t)v_{\ell'}(t)
-
\bigl(\textstyle\sum_{t\ge 1}\delta_p(t)v_\ell(t)\bigr)
\bigl(\textstyle\sum_{t\ge 1}\delta_p(t)v_{\ell'}(t)\bigr)
\Bigr),
\]
and the standard joint probability-generating-function identity gives the formula in part (1) of the theorem. The same displayed bounds prove absolute convergence of the outer sum.

It remains to prove the sign. For each prime $p$, define
\[
C_p(\ell,\ell') :=
\sum_{t\ge 1}\delta_p(t)v_\ell(t)v_{\ell'}(t)
-
\Bigl(\sum_{t\ge 1}\delta_p(t)v_\ell(t)\Bigr)
\Bigl(\sum_{t\ge 1}\delta_p(t)v_{\ell'}(t)\Bigr),
\]
so that $\Cov(v_\ell(\Tam(E)),v_{\ell'}(\Tam(E)))=\sum_p C_p(\ell,\ell')$. We shall prove the stronger local statement
\begin{equation}\label{eq:local-cov-negative}
C_p(\ell,\ell')<0
\qquad\text{for every prime }p\text{ and every }\ell\ne\ell'.
\end{equation}

The geometric tails in Theorem~\ref{thm:GOT-local} give a positive constant $\alpha_p$ with
\begin{equation}\label{eq:tail-alpha}
\delta_p(t)=\alpha_p p^{-t}\qquad(t\ge 5),
\end{equation}
namely
\[
\alpha_2=\frac{1}{2\cdot 1023},\qquad
\alpha_3=\frac{1}{3\cdot 29524},\qquad
\alpha_p=\frac{p^8(p-1)^2}{2(p^{10}-1)}\quad(p\ge 5).
\]
We shall repeatedly use the bounds
\begin{equation}\label{eq:alpha-bounds}
\alpha_p<\tfrac12
\qquad\text{and}\qquad
\alpha_p\ge \tfrac{1}{88572}
\qquad(\text{all primes }p).
\end{equation}
These are immediate for $p=2,3$ from the displayed values. For $p\ge5$, the upper bound follows from $p^8(p-1)^2<p^{10}-1$, and the lower bound follows from $\alpha_p>(p-1)^2/(2p^2)\ge 16/50$.

We split each local sum into a head and tail by
\[
A_{p,r}:=\sum_{1\le t\le 4}\delta_p(t)v_r(t),
\qquad
B_{p,r}:=\sum_{t\ge 5}\delta_p(t)v_r(t).
\]
Since no integer $t\le 4$ is divisible by two distinct primes, the cross-moment is supported entirely on the tail:
\[
\sum_{t\ge 1}\delta_p(t)v_\ell(t)v_{\ell'}(t)
=
\sum_{t\ge 5}\delta_p(t)v_\ell(t)v_{\ell'}(t).
\]
The head sums $A_{p,r}$ admit the lower bounds
\begin{equation}\label{eq:A2A3-lower}
A_{p,2}\ge \frac{1}{2p^2},
\qquad
A_{p,3}\ge \frac{1}{3p^3}
\qquad(p\text{ prime}).
\end{equation}
For $p=2,3$ this is checked directly from the displayed formulas in Theorem~\ref{thm:GOT-local}:
\[
A_{2,2}\ge\delta_2(2)=\frac{7495}{24552}>\frac18,\qquad
A_{2,3}=\delta_2(3)=\frac{1153}{16368}>\frac1{24},
\]
and
\[
A_{3,2}\ge\delta_3(2)=\frac{509345}{6377184}>\frac1{18},\qquad
A_{3,3}=\delta_3(3)=\frac{30619}{2391444}>\frac1{81}.
\]
For $p\ge5$, the first inequality follows from
\[
\delta_p(2)\ge \frac{1}{2p^2},
\]
which is equivalent to
\[
p^{10}-p^8-p^7-p^5-2p^2-2p-1\ge0.
\]
This last polynomial is positive for $p\ge5$ because
$p^5(p^5-p^3-p^2-1)>2p^2+2p+1$. Similarly, the second inequality follows from
\[
\delta_p(3)\ge \frac{1}{3p^3},
\]
which is equivalent to
\[
p^9-2p^8-2p^7-2p^6+p^5-2p^4-2p^3-2p^2-2p-2\ge0;
\]
for $p\ge5$ this is positive since
$p^6(p^3-2p^2-2p-2)>2(p^4+p^3+p^2+p+1)$.

We use the following tail bound. If $r\ne s$ are primes and $q\ge 2$, then
\begin{equation}\label{eq:cross-tail-bound}
\sum_{t\ge 5}q^{-t}v_r(t)v_s(t)
\le
\frac{q^{-rs}}{(1-q^{-rs})^2}.
\end{equation}
The summand vanishes unless $rs\mid t$; and if $i=v_r(t)\ge 1$ and $j=v_s(t)\ge 1$, then $ij\le r^{i-1}s^{j-1}\le t/(rs)$, so
\[
\sum_{t\ge 5}q^{-t}v_r(t)v_s(t)
\le
\sum_{rs\mid t}q^{-t}\frac{t}{rs}
=
\sum_{k\ge1}kq^{-rsk}
=
\frac{q^{-rs}}{(1-q^{-rs})^2}.
\]
For $r\ge 5$ the single tail term $t=r$ gives the complementary lower bound
\begin{equation}\label{eq:B-lower}
B_{p,r}\ge \alpha_p p^{-r}.
\end{equation}

We now prove \eqref{eq:local-cov-negative}. By symmetry assume $\ell<\ell'$.

First let $(\ell,\ell')=(2,3)$. By \eqref{eq:cross-tail-bound},
\[
\sum_{t\ge 1}\delta_p(t)v_2(t)v_3(t)
\le
\alpha_p\frac{p^{-6}}{(1-p^{-6})^2}.
\]
On the other hand, \eqref{eq:A2A3-lower} gives
$A_{p,2}A_{p,3}\ge 1/(6p^5)$. Thus $C_p(2,3)<0$ follows once
\[
\alpha_p\frac{p^{-6}}{(1-p^{-6})^2}<\frac{1}{6p^5}.
\]
For $p\ge5$ this follows from $\alpha_p<1/2$, since
$3/(p(1-p^{-6})^2)<1$. For $p=2$, the required inequality is
\[
\frac{1}{2046}\frac{2^{-6}}{(1-2^{-6})^2}<\frac{1}{6\cdot 2^5},
\]
which is equivalent to $6\cdot 2^5\cdot 2^6<2046\cdot 63^2$. For $p=3$, it is
\[
\frac{1}{88572}\frac{3^{-6}}{(1-3^{-6})^2}<\frac{1}{6\cdot 3^5},
\]
which is equivalent to $6\cdot 3^5\cdot 3^6<88572\cdot 728^2$. Both integer inequalities are immediate.

Next suppose $\ell=2$ and $\ell'\ge5$. By \eqref{eq:cross-tail-bound} and \eqref{eq:B-lower},
\[
\sum_t\delta_p(t)v_2(t)v_{\ell'}(t)
\le
\alpha_p\frac{p^{-2\ell'}}{(1-p^{-2\ell'})^2},
\qquad
A_{p,2}B_{p,\ell'}\ge \frac{\alpha_p}{2p^{\ell'+2}}.
\]
Since
\[
\frac{2p^{2-\ell'}}{(1-p^{-2\ell'})^2}\le
\frac{2\cdot 2^{-3}}{(1-2^{-10})^2}<1,
\]
we have $\sum_t\delta_p(t)v_2(t)v_{\ell'}(t)<A_{p,2}B_{p,\ell'}$, and hence
$C_p(2,\ell')<0$.

The case $\ell=3$ and $\ell'\ge 5$ is analogous: we have
\[
\sum_t\delta_p(t)v_3(t)v_{\ell'}(t)
\le
\alpha_p\frac{p^{-3\ell'}}{(1-p^{-3\ell'})^2},
\qquad
A_{p,3}B_{p,\ell'}\ge \frac{\alpha_p}{3p^{\ell'+3}},
\]
and
\[
\frac{3p^{3-2\ell'}}{(1-p^{-3\ell'})^2}\le
\frac{3\cdot 2^{-7}}{(1-2^{-15})^2}<1,
\]
so $C_p(3,\ell')<0$.

Finally suppose $\ell,\ell'\ge 5$. Then $A_{p,\ell}=A_{p,\ell'}=0$, and \eqref{eq:cross-tail-bound}, \eqref{eq:B-lower}, and \eqref{eq:alpha-bounds} give
\[
\sum_t\delta_p(t)v_\ell(t)v_{\ell'}(t)
\le
\alpha_p\frac{p^{-\ell\ell'}}{(1-p^{-\ell\ell'})^2}
<
\alpha_p^2 p^{-\ell-\ell'}
\le
B_{p,\ell}B_{p,\ell'}.
\]
The middle inequality follows from
\[
\frac{p^{-(\ell\ell'-\ell-\ell')}}{(1-p^{-\ell\ell'})^2}
\le
\frac{2^{-23}}{(1-2^{-35})^2}
<
\frac{1}{88572}
\le \alpha_p,
\]
since for distinct primes $\ell,\ell'\ge 5$ the value $\ell\ell'-\ell-\ell'$ is minimized at $(\ell,\ell')=(5,7)$, where it equals $23$. Summing \eqref{eq:local-cov-negative} over $p$ gives strict negativity of the global covariance.

We now turn to the decay bound in part (3). Write
\[
\mu_{p,r}:=\sum_{t\ge 1}\delta_p(t)v_r(t).
\]
Since $C_p(\ell,\ell')<0$ and both terms in the expression for $C_p$ are nonnegative, we have $|C_p(\ell,\ell')|\le \mu_{p,\ell}\mu_{p,\ell'}$. For $p\ge 5$, the explicit densities and the geometric tail in Theorem~\ref{thm:GOT-local} give
\[
\mu_{p,2}\ll p^{-2},\qquad
\mu_{p,3}\ll p^{-3},\qquad
\mu_{p,r}\ll p^{-r}\quad(r\ge 5).
\]
For $p\in\{2,3\}$, the same geometric tails give $\mu_{p,r}\ll p^{-r}$ for $r\ge 5$, while $\mu_{p,2}$ and $\mu_{p,3}$ are absolute constants. Splitting into the four cases $(\ell,\ell')=(2,3)$, $(2,r)$ with $r\ge 5$, $(3,r)$ with $r\ge 5$, and $\ell,\ell'\ge 5$ distinct, we obtain respectively
\[
\sum_p\mu_{p,2}\mu_{p,3}\ll1,
\]
\[
\sum_p\mu_{p,2}\mu_{p,r}\ll 2^{-r}+3^{-r}+\sum_{p\ge5}p^{-r-2}\ll r^{-2},
\]
\[
\sum_p\mu_{p,3}\mu_{p,r}\ll 2^{-r}+3^{-r}+\sum_{p\ge5}p^{-r-3}\ll r^{-2},
\]
and
\[
\sum_p\mu_{p,\ell}\mu_{p,\ell'}
\ll
2^{-\ell-\ell'}+3^{-\ell-\ell'}+\sum_{p\ge5}p^{-\ell-\ell'}
\ll
(\ell\ell')^{-2}.
\]
After increasing the implied constant to cover the fixed pair $(2,3)$, this proves
part (3).

For $\ell=2$ and $\ell'=3$, direct evaluation of the absolutely convergent series
in part (1), using the explicit formulas of Theorem~\ref{thm:GOT-local} and the
tail bound $\sum_{p>Y}|C_p(2,3)|\ll\sum_{p>Y}p^{-5}$, gives
\[
\Cov(v_2(\Tam(E)),v_3(\Tam(E)))=-0.0248897\ldots,
\]
which is the quoted numerical value.
\end{proof}

\subsection{Proof of the correlation theorem}
We compute the covariance of $\omega_{\Tam}(E)$ and $\Omega(\Tam(E))$ from their joint limiting law, and the variance of $\Omega(\Tam(E))$ in the same way; the correlation is their quotient.

\begin{proof}[Proof of Theorem~\ref{thm:corr}]
Set $\Lambda=\varnothing$, $\Pi=\varnothing$, $s=0$ in Theorem~\ref{thm:master} while keeping $u$ and $w$ free. The resulting joint probability generating function of $(\omega_{\Tam}(E),\Omega(\Tam(E)))$ is
\[
G(u,w)
=
\prod_p g_p(u,w),
\qquad
g_p(u,w):=\delta_p(1)+u\sum_{t\ge 2}\delta_p(t)w^{\Omega(t)}.
\]
Each $g_p$ is a bivariate power series with nonnegative coefficients summing to $1$, so $|g_p|\le 1$ on the closed unit bidisc. To justify differentiation in a neighborhood of $(1,1)$, put
\[
a_p:=1-\delta_p(1),\qquad
T_p:=\sum_{t\ge2}\delta_p(t)t^3.
\]
The geometric tails in Theorem~\ref{thm:GOT-local} show that $T_p<\infty$ for every $p$, and Lemma~\ref{lem:local-estimates} gives $T_p\ll p^{-2}$ for $p\ge5$. On the bidisc $|u-1|,|w-1|\le1$, we have $|u|,|w|\le2$ and $2^{\Omega(t)}\le t$. Consequently,
\[
|g_p(u,w)-1|
\le a_p+2\sum_{t\ge2}\delta_p(t)t
\le3T_p.
\]
Moreover, since $\Omega(t)\le t$, the same majorant gives
\[
|\partial_u g_p|+|\partial_w g_p|
+|\partial_u\partial_w g_p|+|\partial_w^2g_p|
\ll T_p
\]
uniformly on this bidisc. The local series and their displayed derivatives converge uniformly there, and their sums over $p$ converge uniformly as well (using $g_p-1$ for the zeroth derivative). For all sufficiently large $p$, the bound $|g_p-1|\le3T_p$ is at most $1/2$ throughout the bidisc. For the remaining finitely many primes, shrink to a common neighborhood of $(1,1)$ on which $|g_p-1|<1/2$, using $g_p(1,1)=1$. The logarithms normalized by $\log g_p(1,1)=0$ and their first and second derivatives then form uniformly convergent series on a smaller closed bidisc. This justifies termwise logarithmic differentiation of $\log G(u,w)=\sum_p\log g_p(u,w)$ at $(1,1)$.

Using $g_p(1,1)=1$, $\partial_u g_p(1,1)=\sum_{t\ge 2}\delta_p(t)=1-\delta_p(1)$, $\partial_w g_p(1,1)=\sum_{t\ge 1}\delta_p(t)\Omega(t)$, and $\partial_u\partial_w g_p(1,1)=\sum_{t\ge 1}\delta_p(t)\Omega(t)$ (since $\Omega(1)=0$), the joint-p.g.f.~identity gives
\begin{align*}
\Cov(\omega_{\Tam}(E),\Omega(\Tam(E)))
&=
\sum_p\Bigl(\partial_u\partial_w g_p-\partial_u g_p\cdot\partial_w g_p\Bigr)_{(1,1)}\\
&=
\sum_p\Bigl(\sum_{t\ge 1}\delta_p(t)\Omega(t)-(1-\delta_p(1))\sum_{t\ge 1}\delta_p(t)\Omega(t)\Bigr)\\
&=
\sum_p\delta_p(1)\sum_{t\ge 1}\delta_p(t)\Omega(t),
\end{align*}
which is \eqref{eq:cov-omega-Omega}. Absolute convergence follows from Lemma~\ref{lem:local-estimates}, since $\sum_{t\ge 1}\delta_p(t)\Omega(t)\ll p^{-2}$ and $\delta_p(1)\le 1$.

The variance of $\omega_{\Tam}(E)$ is given by Theorem~\ref{thm:omegaE}. Applying the same logarithmic differentiation argument to $G(1,w)$ gives
\[
\Var(\Omega(\Tam(E)))
=
\sum_p\Bigl(\sum_{t\ge 1}\delta_p(t)\Omega(t)^2-\Bigl(\sum_{t\ge 1}\delta_p(t)\Omega(t)\Bigr)^{\!2}\Bigr).
\]
This series converges absolutely by Lemma~\ref{lem:local-estimates}. Both global variances are strictly positive: at $p=2$, the local outcomes $t=1$ and $t=2$ have positive probability and give different values of each local statistic. Dividing the covariance by the square root of the product of the variances proves \eqref{eq:corr-omega-Omega}. Truncating these three convergent sums at $p\le 10^8$ (cf.\ Table~\ref{tab:cutoff-stability}) yields the numerical value $\Corr\approx 0.9801$.
\end{proof}

\subsection{Proof of the moment and tail theorem}
The $k$th moment of the limiting Tamagawa law is the value of its Dirichlet series at $s=-k$. We express it as an Euler product, and the tail bound then follows from Markov's inequality.

\begin{proof}[Proof of Theorem~\ref{thm:tail}]
Fix $k\ge 0$, and write $G_{p,k}(\sigma):=\sum_{t\ge 1}\delta_p(t)t^{k-\sigma}$. The factors at $p=2,3$ are continuous and finite in $\sigma$ by the explicit geometric tails of Theorem~\ref{thm:GOT-local}, and for $p\ge 5$ the proof of \eqref{eq:moment-local-estimate} gives
\[
G_{p,k}(\sigma)=1+O_k(p^{-2})
\qquad (0\le \sigma\le 1),
\]
uniformly in $\sigma$, since $t^{k-\sigma}\le t^k$ for $t\ge 1$. Hence the Euler product $\prod_p G_{p,k}(\sigma)$ converges absolutely and uniformly for $0\le\sigma\le 1$. Applying \eqref{eq:GOT-Euler} at $s=-k+\sigma$ for $0<\sigma\le 1$ gives
\[
\sum_{m\ge 1}P_{\Tam}(m)m^{k-\sigma}
=
\prod_p G_{p,k}(\sigma).
\]
The same uniform convergence gives
\[
\lim_{\sigma\downarrow 0}\prod_p G_{p,k}(\sigma)
=
\prod_p G_{p,k}(0)
=
\prod_p\left(\sum_{t\ge 1}\delta_p(t)t^k\right).
\]
On the other hand, the coefficients $P_{\Tam}(m)m^k$ are nonnegative and $m^{-\sigma}\uparrow 1$ as $\sigma\downarrow 0$, so the monotone convergence theorem gives \cite[Chapter~1]{RudinRCA}
\[
\lim_{\sigma\downarrow 0}\sum_{m\ge 1}P_{\Tam}(m)m^{k-\sigma}
=
\sum_{m\ge 1}P_{\Tam}(m)m^k
=
M_k.
\]
Comparing the last two displays proves \eqref{eq:moment-product}.

For the tail bound, since $(P_{\Tam}(m))_{m\ge 1}$ is a probability distribution on the positive integers, we have
\[
\sum_{m\ge M}P_{\Tam}(m)
\le
\sum_{m\ge M}P_{\Tam}(m)\frac{m^k}{M^k}
\le
\frac{1}{M^k}\sum_{m\ge 1}P_{\Tam}(m)m^k
=
\frac{M_k}{M^k}.
\]
This is \eqref{eq:tail-bound}.
\end{proof}

\subsection{Proofs of the corollaries}

Here we prove the corollaries stated in the introduction.

\begin{proof}[Proof of Corollary~\ref{cor:first-probabilities}]
Expand the Euler product in \eqref{eq:omegaE-gf} and read off coefficients of powers of $u$.
\end{proof}

\begin{proof}[Proof of Corollary~\ref{cor:ell-adic-mean}]
Apply Theorem~\ref{thm:valuations} with $\Pi=\{\ell\}$, and write
\[
J_\ell(z):=\prod_p\left(\sum_{t\ge 1}\delta_p(t)z^{v_\ell(t)}\right).
\]
For $0<z\le 1$ we have
\[
\frac{d}{dz}\sum_{t\ge 1}\delta_p(t)z^{v_\ell(t)}
=
\sum_{t\ge 1}\delta_p(t)v_\ell(t)z^{v_\ell(t)-1},
\]
which extends continuously to $z=0$. By Lemma~\ref{lem:local-estimates}, we have
\[
\sup_{0\le z\le 1}\left|\frac{d}{dz}\sum_{t\ge 1}\delta_p(t)z^{v_\ell(t)}\right|
\le
\sum_{t\ge 1}\delta_p(t)v_\ell(t)
\ll p^{-2}
\qquad (p\ge 5).
\]
Hence the same argument used in the proof of Theorem~\ref{thm:Omega} shows that $J_\ell$ is differentiable at $z=1$ and that
\[
J_\ell'(1)=\sum_p\sum_{t\ge 1}\delta_p(t)v_\ell(t).
\]
Since $J_\ell(z)$ is the probability generating function for $v_\ell(\Tam(E))$, this derivative is exactly the desired expectation.
\end{proof}

\begin{proof}[Proof of Corollary~\ref{cor:joint-omega-tam}]
Take $\Lambda=\varnothing$, $\Pi=\varnothing$, and set $w=1$ in Theorem~\ref{thm:master}.
\end{proof}

\section{Appendix: AxiomProver}\label{sec:AxiomProver}\noindent

AxiomProver is an AI system for mathematical research through formal proof,
currently under development at Axiom Math. The theorems and lemmas in this paper have been autoformalized in Lean by AxiomProver, with 21 formal targets. The repository contains
the formal statements, proofs, and instructions for checking the certificate:
\begin{center}
\url{https://github.com/AxiomMath/BSDTamagawa}
\end{center}
The challenge file gives the definitions and hypotheses.
The solution proves the stated targets. The repository reports
successful verification with the Lean comparator.

\section*{Declaration of generative AI and AI-assisted technologies}

The original manuscript was written entirely by the human authors, who proved all of its theorems and lemmas. AxiomProver (Axiom Math) was used to autoformalize all of the theorems and lemmas in Lean, as described in Section~\ref{sec:AxiomProver} and in the repository linked there. ChatGPT was used for the final editing. The human authors retain full responsibility for the mathematical content, the correspondence between the manuscript and its formalization, and the accuracy and integrity of the final paper.

\end{document}